\documentclass[12pt]{amsart}

\usepackage[latin1]{inputenc}
\usepackage{amsmath}
\usepackage{amsfonts}
\usepackage{amssymb}
\usepackage{graphics,xcolor}
\usepackage{enumerate}
\usepackage{amssymb,amsmath,amsthm,amscd,epsf,latexsym,verbatim,graphicx,amsfonts}
\input epsf.tex

\newtheorem{theorem}{Theorem}
\newtheorem{lemma}{Lemma}
\newtheorem{corollary}[theorem]{Corollary}

\newtheorem{proposition}{Proposition}[section]

\newtheorem{question}{Question}

\newcommand{\bbZ}{\mathbb{Z}}
\newcommand{\bbR}{\mathbb{R}}
\newcommand{\bbC}{\mathbb{C}}

\newcommand{\bbN}{\mathbb{N}}

\newcommand{\bbQ}{\mathbb{Q}}

\newcommand{\mca}{\mathcal{A}}
\newcommand{\mcb}{\mathcal{B}}
\newcommand{\mch}{\mathcal{H}}
\newcommand{\mcj}{\mathcal{J}}

\newcommand{\mcl}{\mathcal{L}}

\newcommand{\mco}{\mathcal{O}}

\newcommand{\mcs}{\mathcal{S}}
\newcommand{\mcu}{\mathcal{U}}

\newcommand{\supp}{\textrm{supp}}
\newcommand{\spm}{\textrm{sp}}

\newcommand{\nri}{n\rightarrow\infty}

\begin{document}
\title[ ] {Spectra of Cayley graphs of the lamplighter group and random Schr\"{o}dinger operators}

\bibliographystyle{plain}

\thanks{  }


\maketitle




\begin{center}
\textbf{Rostislav Grigorchuk $\&$ Brian Simanek}
\end{center}


\begin{abstract}
We show that the lamplighter group $\mcl=\bbZ/2\bbZ\wr\bbZ$ has a system of  generators  for  which  the  spectrum of the  discrete  Laplacian on  the  Cayley  graph  is a  union  of  an interval  and  a  countable  set  of  isolated  points  accumulating  to a point outside this  interval.  This  is  the  first example  of a  group  with  infinitely  many  gaps  in  the  spectrum of  Cayley  graph. The result is obtained by a careful study of spectral  properties  of  a one-parametric  family  $a+a^{-1}+b+b^{-1} - \mu c$ of  convolution operators  on  $\mcl$  where  $\mu$  is a  real  parameter.

Our results show that the spectrum is a pure point spectrum for each value of  $\mu$, the eigenvalues are solutions  of  algebraic equations  involving  Chebyshev polynomials  of  the second kind,  and  the  topological  structure of the spectrum  makes a bifurcation when  parameter  $\mu$  passes  the points  $1$  and  $-1$.  Namely  if  $|\mu| \leq 1$  the  spectrum  is  the  interval  while  when  $|\mu| >  1$ it  is a  union  of  the  interval  and  a  countable  set  of points accumulating  to a point  outside  the  interval.
\end{abstract}

\vspace{5mm}

\section{Introduction}\label{intro}

The study of spectra of finitely generated (non-commutative) groups is a challenging problem initiated by Kesten in \cite{Kesten} and related to many topics in mathematics.  Let $G$ be a group generated by the set $S\subset G$.  By the spectrum of $G$, we mean the spectrum of its Cayley graph $\Gamma(G,S)$, i.e. the spectrum of the Markov operator $M$ in $\ell^2(G)$ corresponding to a simple random walk on $G$ (that is, the random walk given by equal probabilites of the generators and their inverses).  Equivalently, one can think about the spectrum of the discrete Laplace operator $\Delta=I-M$, where $I$ is the identity operator.  A more general point of view is to consider spectra of all Markov operators $M_P$ corresponding to symmetric probability distributions $P$ on the set of generators and its inverses.  Even more informative invariants are the spectral measures $\nu_P$ associated with $M_P$ or the closely related spectral distribution function $N_P(x)$ associated with $\Delta_P$, which allows one to not only determine the spectrum but also calculate the probabilities $\{P_{1,1}^{(n)}\}_{n\in\bbN}$ of the return to the identity $1\in G$ and some other asymptotic characteristics of a group.

Currently, little is known about the possible shape of the spectrum of $M$ (denoted $\mathrm{sp}(M)$) as a set or about the possibilities for the decomposition of the spectral measure $\nu$ into its absolutely continuous, singular continuous, and pure point components.  Also very little is known about how $N(x)$ behaves near $0$ or how $\spm(M)$ depends on the generating set $S$.  For example, it is unknown if $\spm(M)$ can be a Cantor set or if $\nu$ can simultaneously have singular and absolutely continuous parts.

An even larger family of operators in $\ell^2(G)$ given by convolutions of operators determined by elements of the group algebra $\bbC[G]$ or even operators in $(\ell^2(G))^n$ of multiplication by matrices $A\in M_n(\bbC[G])$ can be taken into account.  Questions about their spectra, spectral measures, and other asymptotic characteristics are of great importance in many areas of mathematics such as Novikov-Shubin invariants, Atiyah's $L^2$-Betti numbers, etc. 

In \cite{GZ01}, A. Zuk and the first author showd that the Lamplighter group $\mcl=\bbZ/2\bbZ\wr\bbZ$ has a generating set $\{a,b\}$ with respect to which the spectrum of $M$ is pure point.  Furthermore, the eigenvalues are of the form $\cos\frac{p}{q}\pi$, where $q=2,3,\ldots$, $1\leq p<q$, and $(p,q)=1$.  These eigenvalues densely pack the interval $[-1,1]$ and so as a set, $\spm(M)=[-1,1]$.  The spectral measure $\nu$ is discrete with the mass at $\cos\frac{p}{q}\pi$ equal to $(2^q-1)^{-1}$.  This was the first example of a group and generating set with pure point spectrum of a Markov operator.

In \cite{G15,GV}, Grabowski and Virag observed that if one considers not the operator of convolution with $\frac{1}{4}\left(a+a^{-1}+b+b^{-1}\right)\in\bbZ[\mcl]$ (which is a Markov operator for $\Gamma(\mcl,\{a,b\})$), but instead with $\frac{1}{2+\beta}(a+a^{-1}+\beta c)$ corresponding to the anisotropic random walk given by the distribution $P$ such that $P(a)=P(a^{-1})=\frac{1}{2+\beta}$ and $P(c)=\frac{\beta}{2+\beta}$, where $c=b^{-1}a$ and $\beta\in\bbR_+$, then for large values of $\beta$, the spectral measure $\nu$ is a purely singular continuous measure.  Thus, there is a symmetric random walk on $\mcl$ with a Markov operator that has singular continuous spectrum.  The proof of this latter fact is based on the reduction to the case of random Schr\"{o}dinger operators and a result of Martinelli and Micheli \cite{MM}.

The system $\{a,c\}$ of generators of $\mcl$ is a natural one because of the algebraic structure of $\mcl$ as a semi-direct product:
\[
\mcl=\left(\bigoplus_{\bbZ}\bbZ/2\bbZ\right)\rtimes \bbZ
\]
where a generator $a$ of the ``active" group $\bbZ$ acts on the abelian group  given by the direct sum as the automorphism induced by the shift in the index set $\bbZ$.  The generator $c$ then corresponds to the element $(\ldots,0,0,1,0,0,\ldots)\in\bigoplus_{\bbZ}\bbZ/2\bbZ$.  The generators $\{a,b\}$ correspond to the states of a Mealy type automaton machine $\mca$ over a binary alphabet as shown in Figure 1.

\begin{figure}
  \centering
  \includegraphics[width=0.55\textwidth]{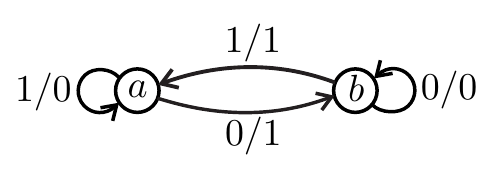}
\caption{The automaton realization of $\mcl$.}
\label{fig:1}
\end{figure}

It is well known that to each invertible automaton $\mcb$ over a finite alphabet, one can associate a group $G(\mcb)=\langle \mcb_{q_1},\ldots,\mcb_{q_m}\rangle$, where $q_1,\ldots,q_m$ are states of the automaton and the operation is composition of automata \cite{GNS}.

Automaton groups are a class of groups that have been used to solve important problems.  The realization of $\mcl$ as $\mcl=\langle \mca_a,\mca_b\rangle$ gave much new information about the group $\mcl$ and showed that it possesses important self-similarity features.  In particular, it lead to the computation of the spectrum and spectral measure as was noticed earlier. This in turn was used in \cite{GLSZ} to answer one of Atiyah's questions on the existence of closed manifolds with non-integer $L^2$-Betti numbers.  In \cite{GZ01}, the authors started with the convolution operator $M_{\mu}$ corresponding to the element $m_{\mu}=a+a^{-1}+b+b^{-1}-\mu c\in\bbR[\mcl]$ and provided computations that eventually lead to the description of the spectrum and spectral measure of $M_{\mu}$ under the assumption that $\mu=0$.
The goal of this paper is to prove the following theorems.

\begin{theorem}\label{t1}
For $\mu\in\bbR$, let $M_{\mu}$ be defined as above.  For every $\mu\in\bbR$, the operator $M_{\mu}$ has pure point spectrum.  Moreover
\begin{itemize}
\item[(a)] If $|\mu|\leq1$, the eigenvalues of $M_{\mu}$ densely pack the interval $[-4-\mu,4-\mu]$.
\item[(b)]  If $|\mu|>1$, the eigenvalues of $M_{\mu}$ form a countable set that densely packs the interval $[-4-\mu,4-\mu]$ and also has an accumulation point $\mu+2/\mu\not\in[-4-\mu,4-\mu]$.
\end{itemize}
\end{theorem}



When $\mu\in\bbR$, the operator $M_{\mu}$ is self-adjoint, the spectral decomposition $M_{\mu}=\int_{\bbR}\lambda dE_{\mu}(\lambda)$ holds, and one can define a spectral measure $\nu_{\mu}$ as
\[
\nu_{\mu}(B)=\langle E_{\mu}(B)\delta_1,\delta_1\rangle,
\]
where $\{E_{\mu}(B):B \,\mbox{is real Borel}\}$ is a spectral family of projections and $\delta_1$ is a delta function at $1\in\mcl$.  We can also provide a detailed description of the spectral measure $\nu_{\mu}$.

\begin{theorem}\label{t2}
The spectral measure $\nu_{\mu}$ of the operator $M_{\mu}$ is given by
\[
\nu_{\mu}=\frac{1}{4}\delta_{\mu}+\sum_{k=2}^{\infty}\left[\frac{1}{2^{k+1}}\sum_{\{s:G_{k}(s,\mu)=0\}}\delta_s\right],
\]
where
\[
G_{k}(z,\mu)=2^k\left[U_k\left(\frac{-z-\mu}{4}\right)+\mu U_{k-1}\left(\frac{-z-\mu}{4}\right)\right],
\]
$U_k$ is the degree $k$ Chebyshev polynomial of the second kind, and zeros of $G_k$ in the above sum are counted with multiplicities.
\end{theorem}

Our calculation of the spectral measure $\nu_{\mu}$ will proceed by reducing it to a problem of finding zeros of polynomials $G_k(z,\mu)$ when $\mu$ is fixed and counting their multiplicities.  The case $\mu=0$ is relatively easy and was investigated in \cite{GZ01}.  To deal with the general case, we will explore the theory of orthogonal polynomials.

Surprisingly, there is a connection between operators associated with groups and random Schr\"{o}dinger operators.  The first link between them was discovered by L. Grabowski and B. Virag in \cite{GV} (see also \cite{G15,G16,KV17}) and the involved group is the lamplighter that we consider here.  Another example is the paper of D. Lenz, T. Nagnibeda, and the first author \cite{GLN18}.

The method of Grabowski and Virag (which goes back to \cite{G14} and T. Austin in \cite{Austin13}) associates to a convolution operator $T\in\ell^2(G)$, where $G=A\rtimes \Gamma$ is a semidirect product of an abelian group $A$ and a countable group $\Gamma$ acting on $A$ by automorphisms, a random family $\{H_{\omega}\}_{\omega\in\Omega}$ of operators in $\ell^2(\Gamma_{\omega})$ where $\Omega=\hat{A}$ is the Pontryagin dual of $A$ supplied with normalized Haar measure.  In some cases (when $\Gamma\simeq\bbZ$, $A=\bigoplus_{\bbZ}\bbZ/n\bbZ$, and specially chosen elements of a group algebra $\bbC[G]$) the family $\{H_{\omega}\}_{\omega\in\Omega}$ becomes a true random Schr\"{o}dinger operator.  The remarkable fact observed in \cite{G14} is that the spectral measure of the operator $T$ associated with the delta function $\delta_1$ at the identity coincides with the density of states of $H_{\omega}$  In \cite{GV} this fact was used to show that some convolution operators on $\mcl$ have continuous singular spectral measure.  This is the first such example in abstract harmonic analysis on discrete groups.

Using the above approach we produce interesting examples for the theory of random Jacobi-Schr\"{o}dinger operators in the Anderson model (i.e. when $(\Omega,\nu)$ is a Bernoulli system).  Let $\Omega=\{0,1\}^{\bbZ}$, $\nu$ be a uniform Bernoulli measure on $\Omega$ (i.e. $\nu=\{1/2,1/2\}^{\bbZ}$), $f$ and $g$ functions defined as
\[
f(\omega)=1+(-1)^{\omega_{-1}},\qquad\qquad g(\omega)=\mu(-1)^{\omega_0+1},
\]
where $\omega_n$ denotes the $n^{th}$ entry of $\omega\in\Omega$.  Let $\{H_{\omega}\}_{\omega\in\Omega}$ be a random family of operators in $\ell^2(\bbZ)$ given by
\begin{equation}\label{anderson}
(H_{\omega}u)(n)=f(T^n\omega)u(n-1)+g(T^n\omega)u(n)+f(T^{n+1}\omega)u(n+1)
\end{equation}
Then as a corollary of Theorems \ref{t1} and \ref{t2} we get the following result.

\begin{theorem}\label{ids}
\begin{itemize}
\item[a)]  The random operator $H_{\mu,\omega}$ almost surely has pure point spectrum and as a set, the spectrum of $H_{\mu,\omega}$ is almost surely given by the set described in Theorem \ref{t1}.
\item[b)]  The density of states of $H_{\mu,\omega}$ is discrete and coincides with the spectral measure described in Theorem \ref{t2}.
\end{itemize}
\end{theorem}

It is known that the density of states in the classical Anderson model is a continuous measure.  We see that for Jacobi-type operators this is no longer true.  Also, Theorem \ref{t2} gives us the exact description of the density of states.  We will discuss this more in Section \ref{schro}.

The final section discusses further topics related to the results of this paper, in particular the Novikov-Shubin invariants.  The last example at the very end of the text shows that $\mcl$ has operators of convolution given by the elements of the group algebra with rational (or even integer) coefficients and irrational Novikov-Shubin invariants.  The first such examples were given in \cite{G16}.

\section{Preliminaries}\label{prelim}

In this section we will introduce some relevant information about orthogonal polynomials and their relationship to the spectral theory of Jacobi matrices.  In particular we will be interested in understanding the location of the zeros of orthogonal polynomials in terms of the spectrum of the corresponding Jacobi matrix.  Also we will provide a short background to the spectral theory of groups acting on rooted trees.

\subsection{Orthogonal Polynomials on the Real Line}\label{op}

Consider a probability measure $\gamma$ with compact and infinite support in the real line.  By performing Gram-Schmidt orthogonalization on the sequence of monomials, one arrives at the sequence of orthonormal polynomials $\{\varphi_n(x)\}_{n=0}^{\infty}$, where the degree of $\varphi_n$ is exactly $n$ and the leading coefficient of $\varphi_n$ is positive.  By dividing each polynomial by its leading coefficient, we obtain the sequence $\{p_n(x)\}_{n=0}^{\infty}$ of monic polynomials that are still mutually orthogonal in the space $L^2(\gamma)$.  The most basic facts about the zeros of $p_n$ are (see \cite[Section 1.2.5]{OPUC1}):
\begin{itemize}
\item  The zeros of $p_n$ are all real and simple.
\item  Between any two zeros of $p_{n+1}$ there is a zero of $p_n$.
\end{itemize}

To get more detailed information about the zeros of $p_n$, we need to employ ideas from the spectral theory of bounded self-adjoint operators.  The polynomials $\{p_n\}_{n=0}^{\infty}$ satisfy a three-term recurrence relation
\begin{equation}\label{3term}
p_{n+1}(x)=(x-b_{n+1})p_n(x)-a_n^2p_{n-1}(x),\qquad\qquad n=0,1,2,\ldots,
\end{equation}
where $p_{-1}=0$ by convention.  In \eqref{3term}, $b_n\in\bbR$ and $a_n>0$ for each $n\in\bbN$.  The orthonormal polynomials also satisfy a three-term recurrence relation, which is most easily expressed in the following formal matrix notation:
\begin{equation}\label{jac1}
\begin{pmatrix}
b_1 & a_1 & 0 & 0 & \cdots\\
a_1 & b_2 & a_2 & 0 & \cdots\\
0 & a_2 & b_3 & a_3 & \cdots\\
0 & 0 & a_3 & b_4 & \ddots\\
\vdots & \vdots & \vdots & \ddots & \ddots
\end{pmatrix}
\begin{pmatrix}
\varphi_0(x)\\
\varphi_1(x)\\
\varphi_2(x)\\
\varphi_3(x)\\
\vdots
\end{pmatrix}=
\begin{pmatrix}
x\varphi_0(x)\\
x\varphi_1(x)\\
x\varphi_2(x)\\
x\varphi_3(x)\\
\vdots
\end{pmatrix}
\end{equation}
The tri-diagonal self-adjoint matrix on the far left of \eqref{jac1} is the Jacobi matrix $\mcj$ that defines a self-adjoint operator on $\ell^2(\bbN)$, which we also denote by $\mcj$.  The measure $\gamma$ is the spectral measure for this matrix and the vector $(1,0,0,\ldots)^T$.  
Consequently, the support of $\gamma$ is the spectrum of $\mcj$.  This relationship allows us to connect the zeros of $p_n$ to the spectrum of $\mcj$.  Indeed, the following facts can be found in \cite[Section 1.2.11]{OPUC1}:
\begin{itemize}
\item  If $y\in\supp(\gamma)$ and $\epsilon>0$, then for all sufficiently large $n$ the polynomial $p_{n}$ has a zero within $\epsilon$ of $y$
\item  If $I$ is an interval that is disjoint from $\supp(\gamma)$, then $p_{n}$ has at most one zero in $I$.
\end{itemize}
These two facts tell us that when $n$ is very large, the collection of zeros of $p_n$ very closely resembles the support of the measure $\gamma$.

The above discussion tells us that as $\nri$, the number of sign changes of $p_n$ on the interval $(x,\infty)$ approaches the cardinality of $\supp(\gamma)\cap(x,\infty)$.  Thus, if $x$ is chosen so that $\supp(\gamma)\cap(x,\infty)$ has finite cardinality $N$, then for all very large $n$, the polynomial $p_n$ will have exactly $N$ zeros in the interval $(x,\infty)$.  Furthermore, by the interlacing property of the zeros of orthogonal polynomials, the limit points of these $N$ zeros (as $\nri$) are precisely the $N$ points in $\supp(\gamma)\cap(x,\infty)$ and the approach to these points is monotonic.

One very well understood example is the case when $a_n\equiv1$ and $b_n\equiv0$.  The Jacobi matrix in this case is called the free Jacobi matrix and can be written as $\mathfrak{L}+\mathfrak{R}$, where $\mathfrak{L}$ is the left shift operator on $\ell^2(\bbN)$ and $\mathfrak{R}$ is the right shift operator on $\ell^2(\bbN)$.  In this case, the spectrum of the Jacobi matrix is $[-2,2]$ and the spectral measure $\gamma$ is purely absolutely continuous on that interval with weight $\frac{1}{2\pi}\sqrt{4-x^2}$.  The corresponding orthonormal polynomials are rescaled Chebyshev polynomials of the second kind, namely $\{U_n(x/2)\}_{n=0}^{\infty}$, where $U_n$ is the $n^{th}$ Chebyshev polynomial.  The polynomials $\{U_n\}_{n=0}^{\infty}$ have the following properties, which we will use later:
\begin{align}
\label{urecur}U_{n+1}(x)&=2xU_n(x)-U_{n-1}(x)\\
\label{utrig}U_n(\cos(x))&=\frac{\sin((n+1)x)}{\sin(x)}\\
\label{uone}U_n(1)&=n+1\\
\label{uratio}\lim_{\nri}\frac{U_n(x)}{U_{n+1}(x)}&=\frac{1}{x+\sqrt{x^2-1}},\qquad\qquad x\not\in[-1,1]\\
\label{first} U_0(x)=1, \quad & \quad U_1(x)=2x 
\end{align}
where the square root is defined with the branch cut along $[-1,1]$ and so that $\sqrt{x^2-1}>0$ when $x>1$.

\subsection{Matrix Recursions and the Limit Spectral Measures}

A group $G$ defined by an automaton over an alphabet of $d$ letters naturally acts by automorphisms of a $d$-regular rooted tree $T=T_d$ (see \cite{Gr11,GNS}).  The self-similarity structure given by the automaton realization of $G$ leads to self-similarity properties of involved Schreier graphs $\{\Gamma_n\}_{n\in\bbN}$ and $\{\Gamma_{\xi}\}_{\xi\in\partial T}$ associated with the action on the $n^{th}$ level of the tree and the orbit $G\xi$.  If a group $G$ with a generating set $S$ acts transitively on a set $X$, then the Schreier graph (also called the graph of action) consists of the set of vertices $V=X$ and the set of edges $\{(x,sx):x\in X, \, s\in S\}$.

Another set of important operators involved in this study are those of the form $\pi(m)$, where $m\in\bbC[G]$ is an element of the group algebra and $\pi:G\rightarrow\mcu(L^2(\partial T,\gamma))$ is a Koopman representation given by
\[
\pi f(x)=f(g^{-1}x),\qquad\qquad f\in L^2(\partial T,\gamma)
\]
In general, spectra of Schreier graphs $\Gamma_{\xi}$ are contained in the spectrum of $\pi(m)$, where
\begin{equation}\label{mdef}
m=\frac{1}{2|\mcs|}\sum_{s\in\mcs\cup\mcs^{-1}}s
\end{equation}
but in the case when the graphs $\Gamma_{\xi}$ are amenable (and this holds for instance if $G$ is amenable), the relations $\spm(\Gamma_{\xi})=\spm(\pi(m))$ hold (see \cite{BG00}).  Moreover,
\[
\spm(\pi(m))=\overline{\bigcup_{n=0}^{\infty}\spm(\Gamma_n)}
\]
and $\spm(\Gamma_n)\subset\spm(\Gamma_{n+1})$ as $\Gamma_{n+1}$ covers $\Gamma_n$ for each $n$.

There are several examples of computations of the spectra $\spm(\Gamma_n)$, the spectral counting measure $\sigma_n$, and of the limit spectral measures $\sigma_*:=\lim_{\nri}\sigma_n$ provided in \cite{BG00,GZ01} and other articles.  It is known from \cite{GZ04} that this limit exists and it is called there the KNS-spectral measure.  This measure could also be called the density of states because of the similarity of its definition to the classical notion of density of states used in the theory of random Schr\"{o}dinger operators (see \cite{IDS}).

The study of spectra of graphs $\Gamma_n$ and $\Gamma_{\xi}$ are closely related.  By the spectrum of a graph, we mean first of all the spectrum of the adjacency matrix.  The spectrum of $\Gamma$is denoted by $\spm(\Gamma)$.  A more general point of view of the spectrum involves use of weights along the edges of the graph (e.g. anisotropic case).

The meaning of $\sigma_*$ can be understood as follows.  For a graph $\Gamma=(V,E)$ and vertex $v\in V$, let $\sigma_v$ be the spectral measure given by $\sigma_v(B)=\langle E(B)\delta_v,\delta_v\rangle$, where $\delta_v$ is the delta mass at $v$, $B$ is a Borel subset of $\bbR$, and $E(B)$ is a spectral family given by the spectral theorem for $M$.  For a simple random walk on $\Gamma$ that starts at $v$ and with equal transition probabilities along edges, the probability $P_{v,v}^{(n)}$ of return after $n$ steps is the $n^{th}$ moment of the measure $\sigma_v$.  Now apply this reasoning to the family $\{\Gamma_{\xi}\}_{\xi\in\partial T}$ with $\xi$ taken as the initial point of the random walk and take the averages
\[
\tilde{P}^{(n)}=\int P_{\xi,\xi}^{(n)}d\gamma(\xi)
\]
Then $\tilde{P}^{(n)}$ is the $n^{th}$ moment of the measure $\sigma_{*}$ \cite{Gr11}.  Thus if the action $G\curvearrowright \partial T$ on the boundary of the tree is essentially free, then the graphs $\{\Gamma_{\xi}\}_{\xi\in\partial T}$  are almost surely isomorphic to the Cayley graphs $\Gamma(G,S)$ and so the probabilities of return do not depend on the starting vertex (because the Cayley graphs have a transitive group of automorphisms), the probabilities $\{P_{\xi,\xi}^{(n)}\}_{\xi\in\partial T}$ are constant almost surely and coincide with the averaged probability $\tilde{P}^{(n)}$.  Hence the KNS-spectral measure $\sigma_*$ coincides in this case with the Kesten's spectral measure associated with the random walk on $G$.  The paper \cite{GZ01} gave another justification of this fact based on the use of the $C^*$-algebra generated by the Koopman representation and the recurrent trace defined in \cite{GZ01}.  A clearer form of this result is provided in \cite{Gr11}.

The above arguments in fact work not only for a simple random walk (and hence for the operator $\lambda_G(m)$, with $m$ as in \eqref{mdef}), but for operators given by any self-adjoint elements of the group algebra $\bbC[G]$.

\section{Recursive relations and determinants}\label{spec}

Let us return our attention to the lamplighter group $\mcl=\langle a,b\rangle$ and its realization by the automaton $\mca$ from Figure $1$.  The group acts by automorphisms of a rooted binary tree (the action is determined by the automaton structure).  Our approach to the computation of the relevant spectrum will follow that used in \cite[Section 6]{GZ01}, so let us review the set-up presented there.

Let $T=T_2$ be a binary rooted tree and let $\eta$ denote the uniform Bernoulli measure on the boundary of $T$, $\partial T=\{0,1\}^{\bbN}$.  Let $\mch$ denote the Hilbert space $L^2(\partial T,\eta)$.  The left and right branches $T_0$ and $T_1$ of $T$ are canonically isomorphic to the whole tree $T$ and the restrictions of $\eta$ to each $T_i$ ($i=0,1$) are - after appropriate normalization - equal to $\eta$.  This leads to the self-similarity  $\mch\cong\mch\oplus\mch$ of the Hilbert space $\mch$ and to the operator recursions
\[
\pi(a)=\begin{pmatrix} 0 & \pi(a)\\ \pi(b) & 0\end{pmatrix},\qquad\qquad\pi(b)=\begin{pmatrix}\pi(a) & 0\\ 0 & \pi(b)\end{pmatrix}
\]
where $\pi$ is the Koopman representation (for more on self-similar $C^*$-algebras and operator recursions, see \cite{GN07}).  The operator $\pi(c)$ corresponding to the element $c=b^{-1}a\in\mcl$, which has order $2$, is presented by
\[
\pi(c)=\begin{pmatrix} 0 & I\\ I & 0\end{pmatrix}
\]
where $I$ is the identity operator (because the automorphism $c$ just permutes the vertices of the first level and hence switches  $T_0$ and $T_1$ without further action inside $T_0$ or $T_1$).

Similarly, let $V_n$ be the set of $2^n$ vertices of the $n^{th}$ level of $T$.  The matrices of size $2^n\times2^n$ $a_n$, $b_n$, and $c_n$ presenting generators $a$, $b$, and $c$ by their actions on the space $\ell^2(V_n)$ satisfy the recurrent relations
\[
a_n=\begin{pmatrix} 0 & a_{n-1}\\ b_{n-1} & 0\end{pmatrix}\qquad b_n=\begin{pmatrix}a_{n-1} & 0\\ 0 & b_{n-1}\end{pmatrix}\qquad c_n=\begin{pmatrix}0 & I_{2^{n-1}}\\ I_{2^{n-1}} & 0\end{pmatrix}
\]
The sum $a_n+a_n^{-1}+b_n+b_{n}^{-1}$ is the adjacency matrix of the Schreier graph $\Gamma_n$ when the system of generators $\{a,b\}$ for $\mcl$ is used.  Similarly, the sum $a_n+a_n^{-1}+b_n+b_n^{-1}+c_n$ is the adjacency matrix when the system of generators $\{a,b,c\}$ for $\mcl$ is used.  We include the latter sum in the one-parameter pencil
\[
M_n(\mu)=a_n+a_n^{-1}+b_n+b_n^{-1}-\mu c_n,\qquad\qquad\mu\in\bbC
\]
of matrices of size $2^n$.  We will now see that one can relate the problem of finding the eigenvalues (and their multiplicities) of $M_n(\mu)$ to solving certain equations involving Chebyshev polynomials of the second kind.

Following the notation in \cite{GZ01}, set
\[
S_{n+1}=\begin{pmatrix}
0 & I_{2^n}\\
I_{2^n} & 0
\end{pmatrix}
\]
and let us consider the calculation of
\[
\Phi_n(\lambda,\mu):=\det\left(a_n+a_n^{-1}+b_n+b_n^{-1}-\mu S_n-\lambda I_{2^n}\right).
\]
From this we find
\begin{align*}
\Phi_0(\lambda,\mu)&=(4-\lambda-\mu)\\
\Phi_1(\lambda,\mu)&=(\mu-\lambda)(4-\lambda-\mu)\\
\Phi_2(\lambda,\mu)&=(\mu-\lambda)(4-\lambda-\mu)(\lambda^2-\mu^2-4)\\
\Phi_3(\lambda,\mu)&=(\lambda-\mu)^2(\lambda+\mu-4)(\lambda^2-\mu^2-4)(\lambda^3+\lambda^2\mu-\lambda\mu^2-\mu^3-8\lambda)
\end{align*}
Observe that $\Phi_n$ is a factor of $\Phi_{n+1}$ since the graph $\Gamma_{n+1}$ covers $\Gamma_n$.
The first step in this calculation will involve writing $\Phi_n(\lambda,\mu)$ as a product of many factors of smaller degree.  In \cite[Section 6]{GZ01} it is shown that
\begin{equation}\label{star23}
\Phi_n(\lambda,\mu)=(\mu-\lambda)^{2^n}\Phi_{n-1}(\lambda',\mu'),\qquad n\geq1,
\end{equation}
where
\[
\lambda'=\frac{\lambda^2-\mu^2-2}{\lambda-\mu},\qquad\qquad\mu'=\frac{2}{\lambda-\mu}
\]
and hence one can establish a recursive relationship for the functions $\{\Phi_n(\lambda,\mu)\}_{n\in\bbN}$.  Indeed, one has
\[
\Phi_n(\lambda,\mu)=(\mu-\lambda)^{2^{n-1}+\cdots+2+1}\Phi_0(F^{(n)}(\lambda,\mu))
\]
where $F:\bbC^2\rightarrow\bbC^2$ is given by $\lambda\rightarrow\lambda'$, $\mu\rightarrow\mu'$.  Thus the dynamics of the map $F$ is relevant.  The lines $\ell_c:=\{\lambda+\mu=c\}$ are $F$-invariant since $\lambda+\mu=\lambda'+\mu'$.  As observed in \cite{GZ04}, the restriction of $F$ to $\ell_c$ is conjugate to a modular mapping given by the matrix
\[
Q_c=\begin{pmatrix}
c & -\frac{c^2}{2}-1\\
1 & -\frac{c}{2}
\end{pmatrix}\in SL(2,\bbC)
\]
which is elliptic if $|c|<4$, parabolic of $|c|=4$, and hyperbolic if $|c|>4$.  When $|c|<4$ the restriction $F\big|_{\ell_c}$ is conjugate to the rotation by an angle $\varphi=\arctan\frac{\sqrt{4-c^2}}{c}$ and hence is ``chaotic" when the angle is irrational.   Thus, in the strip $\Omega=\{\ell_c:|c|<4\}$, the behavior of $F$ is partially chaotic.

Outside the strip $\Omega$, the orbit of each point tends to infinity.  Unfortunately, understanding the dynamics of $F$ does not help us find the spectrum of the pencil $M_n(\mu)$.  For this purpose, the important relation
\[
\mu'-\lambda'=-\lambda-\mu+\frac{4}{\lambda-\mu}
\]
is useful.  If we denote $(\lambda^{(n)},\mu^{(n)})=F^{(n)}(\lambda,\mu)$, then
\[
\mu^{(n)}-\lambda^{(n)}=-\lambda^{(n-1)}-\mu^{(n-1)}+\frac{4}{\lambda^{(n-1)}-\mu^{(n-1)}}
\]
which leads us to the relation
\[
\lambda^{(n)}-\mu^{(n)}=\frac{G_n(\lambda,\mu)}{H_n(\lambda,\mu)}
\]
where
\begin{equation}\label{grec}
\begin{pmatrix}
G_{k+1}\\ H_{k+1}
\end{pmatrix}=
\begin{pmatrix}
-\lambda-\mu & -4\\
1 & 0
\end{pmatrix}
\begin{pmatrix}
G_{k}\\ H_{k}
\end{pmatrix}
\end{equation}
and $G_1(\lambda,\mu)=\mu-\lambda$, $H_1(\lambda,\mu)=1$.  Notice that
\begin{equation}\label{hgk}
H_k=G_{k-1},\qquad\qquad k\geq2.
\end{equation}
To proceed with our calculations, we need the following lemma.

\begin{lemma}\label{ghform}
It holds that
\begin{align}
\label{gform}G_{k}(\lambda,\mu)&=2^{k-1}(\mu-\lambda)U_{k-1}\left(\frac{-\lambda-\mu}{4}\right)-2^{k+1}U_{k-2}\left(\frac{-\lambda-\mu}{4}\right)\\
\label{hform}H_{k}(\lambda,\mu)&=2^{k-1}\left[U_{k-1}\left(\frac{-\lambda-\mu}{4}\right)+\mu U_{k-2}\left(\frac{-\lambda-\mu}{4}\right)\right]
\end{align}
with the understanding that $U_{-1}=0$.  
\end{lemma}

\begin{proof}
From the recursion relation \eqref{grec} satisfied by the polynomials $G_k$ and $H_k$ we conclude that
\[
\begin{pmatrix}
G_{k+1}\\ H_{k+1}
\end{pmatrix}=
\begin{pmatrix}
-\lambda-\mu & -4\\
1 & 0
\end{pmatrix}^k
\begin{pmatrix}
\mu-\lambda \\ 1
\end{pmatrix}
\]
Applying \cite[Theorem 1]{McL}, we can rewrite this as
\begin{align*}
\begin{pmatrix}
G_{k+1}\\ H_{k+1}
\end{pmatrix}&=
\begin{pmatrix}
2^kU_k\left(\frac{-\lambda-\mu}{4}\right) & -2^{k+1}U_{k-1}\left(\frac{-\lambda-\mu}{4}\right)\\
2^{k-1}U_{k-1}\left(\frac{-\lambda-\mu}{4}\right) & 2^kU_k\left(\frac{-\lambda-\mu}{4}\right)+2^{k-1}(\mu+\lambda)U_{k-1}\left(\frac{-\lambda-\mu}{4}\right)
\end{pmatrix}
\begin{pmatrix}
\mu-\lambda \\ 1
\end{pmatrix}
\end{align*}
which implies the desired formulas if $k\geq2$.  By inspection, we see that those formulas also holds when $k=1$.
\end{proof}

The relation \eqref{star23} shows that
\begin{equation}\label{phiform}
\Phi_n(\lambda,\mu)=(4-\lambda-\mu)\prod_{k=1}^n\left(\frac{G_k(\lambda,\mu)}{H_k(\lambda,\mu)}\right)^{2^{n-k}}=(4-\lambda-\mu)G_1^{2^{n-2}}G_2^{2^{n-3}}\cdots G_{n-1}G_n
\end{equation}
(see \cite[Section 6]{GZ01}).  
This leads us immediately to the following corollary.

\begin{corollary}\label{phizero}
Using the notation defined above,
\begin{align*}
&\left\{(\lambda,\mu):\Phi_n(\lambda,\mu)=0\right\}=\\
&\quad\{(\lambda,\mu):\lambda+\mu=4\}\bigcup\left\{\bigcup_{k=1}^n\left\{(\lambda,\mu):(\mu-\lambda)U_{k-1}\left(\frac{-\lambda-\mu}{4}\right)=4U_{k-2}\left(\frac{-\lambda-\mu}{4}\right)\right\}\right\}
\end{align*}
where $U_k$ is the degree $k$ Chebyshev polynomial of the second kind.
\end{corollary}


\noindent The zero set of $\Phi_n(\lambda,\mu)$ is pictured below for $n=6$.  We can easily calculate
\begin{align*}
\Phi_6(\lambda,\mu)&=(\lambda-\mu)^{16}(4-\lambda-\mu)(4-\lambda^2+\mu^2)^8(\lambda^3+\lambda^2\mu-\mu^3-\lambda(8+\mu^2))^4\\
&\quad(\lambda^5+16\mu+3\lambda^4\mu+8\mu^3-\mu^5+2\lambda^3(\mu^2-8)-2\lambda^2\mu(12+\mu^2)-3\lambda(\mu^4-16))\\
&\quad(64-(\lambda+\mu)(12-(\lambda+\mu)^2)(\lambda^3+\lambda^2\mu-\mu^3-\lambda(8+\mu^2)))\\
&\quad(16+(\lambda+\mu)(\lambda^3+4\mu+\lambda^2\mu-\mu^3-\lambda(12+\mu^2)))^2
\end{align*}
Since $\Phi_n(\lambda,\mu)$ divides $\Phi_{n+1}(\lambda,\mu)$, the pictured set will be a subset of the zero set of $\Phi_m$ for all $m\geq6$.  Notice that there is a critical value of $\mu$ above which we see a zero of $\Phi_6(\lambda,\mu)$ outside of the strip $\{|\lambda+\mu|\leq4\}$.  This value appears to be close to $\mu=1$, but we will see later that for $n=6$, this critical value is actually $7/6$.  Later we will see how this value is calculated.

\begin{figure}[h!]
  \centering
  \includegraphics[width=0.35\textwidth]{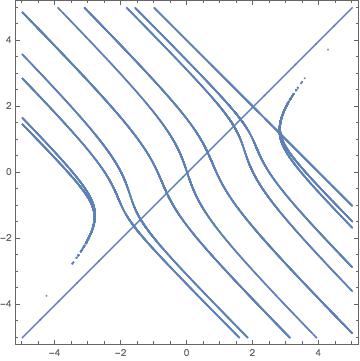}
\caption{The curves showing where $\Phi_6(\lambda,\mu)=0$ when $\lambda$ and $\mu$ are real.  $\lambda$ is the horizontal axis and $\mu$ is the vertical axis.}
\label{fig:2}
\end{figure}

We suspect that for each $n\in\bbN$, the curves defined by the condition $\{(\lambda,\mu):G_n(\lambda,\mu)=0\}$ are irreducible, but we do not have a proof of this fact.

\section{Eigenvalues and their multiplicities}\label{emult}

Our next task is to understand some finer properties of the zero set of $\Phi_n(\lambda,\mu)$.  In particular, if we think of $\mu$ as a fixed constant and $\lambda$ as a variable in $\bbC$, then we want to understand the distribution of zeros of $\Phi_n(\lambda,\mu)$ for large $n$ and also the multiplicities of those zeros.

From \eqref{grec} we see that the polynomials $\{G_k(\lambda,\mu)\}_{k=1}^{\infty}$ satisfy the recurrence relation
\[
G_{k+1}(\lambda,\mu)=(-\lambda-\mu)G_k(\lambda,\mu)-4G_{k-1}(\lambda,\mu).
\]
It follows that
\begin{equation}\label{pgform}
G_k(\lambda,\mu)=2^kP_{k,\mu}(-\lambda/2),
\end{equation}
where $P_{k,\mu}(z)$ is the degree $k$ monic orthogonal polynomial corresponding to the Jacobi matrix
\[
\mcj^*(\mu)=\begin{pmatrix}
-\frac{\mu}{2} & 1 & 0 & 0 &\cdots\\
1 & \frac{\mu}{2} & 1 & 0 &\cdots\\
0 & 1 & \frac{\mu}{2} & 1 & \ddots\\
\vdots & \ddots & \ddots & \ddots & \ddots
\end{pmatrix}
\]
One immediate consequence of \eqref{pgform} is that when $\mu$ is fixed all the zeros of $G_k(\lambda,\mu)$ are simple.  Furthermore, our discussion in Section \ref{prelim} showed the relationship between the zeros of $G_k(\lambda,\mu)$ and the spectrum of the corresponding Jacobi matrix.  Thus we can deduce information about the zeros of $G_k$ if we can identify the spectrum of the operator $\mcj^*(\mu)$.  This is the content of our next result.

\begin{proposition}\label{spectrum}
When $\mu$ is any complex number, it holds that the spectrum of $\mcj^*(\mu)$ is given by the straight line segment joining $-2+\frac{\mu}{2}$ to $2+\frac{\mu}{2}$ and - precisely when $|\mu|>1$ - the isolated point $-\frac{\mu}{2}-\frac{1}{\mu}$.
\end{proposition}

\begin{proof}[First proof of Proposition \ref{spectrum}:]
The spectrum of $\mcj^*(\mu)$ is just a translation (in $\bbC$) by $\mu/2$ of the spectrum of
\[
\mcj^*_1(\mu)=\begin{pmatrix}
-\mu & 1 & 0 & 0 &\cdots\\
1 & 0 & 1 & 0 &\cdots\\
0 & 1 & 0 & 1 & \ddots\\
\vdots & \ddots & \ddots & \ddots & \ddots
\end{pmatrix}
\]
so we will consider the spectrum of this operator instead.  The essential spectrum of $\mcj^*_1(\mu)$ is $[-2,2]$ (see \cite[Equation 2.12]{Beck}), so we need only locate the isolated eigenvalues of $\mcj^*_1(\mu)$.  For this, we will use \cite[Proposition 1]{AKVA}.  If we consider the continued fraction
\[
\cfrac{1}{z+\mu-\cfrac{1}{z-\cfrac{1}{z-\cdots\cfrac{1}{z}}}}=\frac{1}{z+\mu-\frac{A_{n-1}(z)}{B_{n-1}(z)}}=\frac{K_{n}(z)}{L_n(z)},
\]
then $z_0$ is an eigenvalue of $\mcj^*_1(\mu)$ if and only if $\{L_n(z_0)\}_{n\in\bbN}\in\ell^2(\bbN)$.  By \cite[Equation 6]{AKVA}, we know that $B_{n-1}(z)=U_{n-1}(z/2)$ and $A_{n-1}(z)=U_{n-2}(z/2)$.  Thus,
\[
K_n(z)=U_{n-1}(z/2)\qquad\qquad\qquad L_n(z)=(z+\mu)U_{n-1}(z/2)-U_{n-2}(z/2)
\]
so the only possible isolated eigenvalues are values of $z_0\not\in[-2,2]$ for which
\[
\sum_{n=2}^{\infty}|(z_0+\mu)U_{n-1}(z_0/2)-U_{n-2}(z_0/2)|^2<\infty
\]
For such a value of $z_0$, it must certainly be the case that
\[
(z_0+\mu)-\frac{U_{n-2}(z_0/2)}{U_{n-1}(z_0/2)}\rightarrow0
\]
as $\nri$.  Taking $n$ to infinity and invoking (\ref{uratio}) shows
\[
z_0+\mu=\frac{2}{z_0+\sqrt{z_0^2-4}},
\]
and hence $z_0=-\mu-\frac{1}{\mu}$.  Thus, this is the only possible isolated eigenvalue of $\mcj^*_1(\mu)$.
To see that it actually is an eigenvalue, we write $\mu=-e^{ix}$ for some $x\in\bbC$.  Then we are left to evaluate
\[
\sum_{n=2}^{\infty}|e^{-ix}U_{n-1}(\cos(x))-U_{n-2}(\cos(x))|^2
\]
By \eqref{utrig} we can write
\begin{align*}
&e^{-ix}U_{n-1}(\cos(x))-U_{n-2}(\cos(x))=e^{-ix}\frac{\sin(nx)}{\sin(x)}-\frac{\sin((n-1)x)}{\sin(x)}\\
&=e^{-ix}\frac{e^{inx}-e^{-inx}}{e^{ix}-e^{-ix}}-\frac{e^{i(n-1)x}-e^{-i(n-1)x}}{e^{ix}-e^{-ix}}=\frac{e^{-i(n-1)x}-e^{-i(n+1)x}}{e^{ix}-e^{-ix}}=e^{-inx}\frac{e^{ix}-e^{-ix}}{e^{ix}-e^{-ix}}=e^{-inx}.
\end{align*}
Thus the above sum simplifies to
\[
\sum_{n=2}^{\infty}\left|e^{-inx}\right|^2=\sum_{n=2}^{\infty}|\mu|^{-2n}
\]
which is finite if and only if $|\mu|>1$.  After translating by $\mu/2$ to get the spectrum of $\mcj^*(\mu)$, we get the desired conclusion.
\end{proof}

\begin{proof}[Second proof of Proposition \ref{spectrum} for real $\mu$:]
Define
\[
\mcj_0(\mu)=\begin{pmatrix}
\frac{\mu}{2} & 1 & 0 & 0 &\cdots\\
1 & \frac{\mu}{2} & 1 & 0 &\cdots\\
0 & 1 & \frac{\mu}{2} & 1 & \ddots\\
\vdots & \ddots & \ddots & \ddots & \ddots
\end{pmatrix}
\]
whose corresponding spectral measure for the vector $e_0=(1,0,0,\ldots)^T$ is
\[
d\nu_0:=\frac{1}{2\pi}\sqrt{4-(x-\mu/2)^2}\,dx,\qquad\qquad x\in[-2+\mu/2,2+\mu/2]
\]
Since $\mcj^*(\mu)$ is a compact perturbation of $\mcj_0(\mu)$, we know that the essential spectrum of $\mcj^*(\mu)$ is $[-2+\mu/2,2+\mu/2]$.  Thus the spectrum of $\mcj^*(\mu)$ is this interval and possibly a countable set of mass points outside this interval, which we now determine.  

For any probability measure $\gamma$ on $\bbR$, we define
\[
m(z;\gamma)=\int_{\bbR}\frac{1}{x-z}d\gamma(x)
\]
Let $\nu^*$ be the spectral measure of $\mcj^*(\mu)$ and $e_0$.  Then
\[
m(z;\nu^*)=\frac{1}{-\mu/2-z-m(z;\nu_0)}
\]
(see \cite[Theorem 3.2.4]{Rice}).  Also using \cite[Theorem 3.2.4]{Rice}, one finds that
\[
m(z;\nu_0)=\frac{\mu}{4}-\frac{z}{2}-\frac{\sqrt{(z-\mu/2)^2-4}}{2}
\]
where the branch cut is taken to be positive along the interval $[2+\mu/2,\infty)$.  Thus
\[
m(z;\nu^*)=\frac{-4}{2z+3\mu-\sqrt{(2z-\mu)^2-16}}=\frac{-(2z+3\mu+2\sqrt{(z-\frac{\mu}{2})^2-4})}{4(z\mu+\frac{\mu^2}{2}+1)}
\]

We know from \cite[Proposition 2.3.12]{Rice} that the singular part of $\nu^*$ is supported on the set
\[
\{x\in\bbR:\lim_{t\rightarrow0^+}\mathrm{Im}(m(x+it;\nu^*))=\infty\}
\]
We calculate
\begin{align*}
\lim_{t\rightarrow0^+}\mathrm{Im}(m(x+it;\nu^*))&=\frac{\sqrt{(\mu-2x)^2-16}}{2(\mu x+\mu^2/2+1)}
\end{align*}
Thus we see that if $\mu=0$, then there is no singular component to the spectrum and if $\mu\neq0$, then the only possible mass point is $x^*=-\frac{1}{\mu}-\frac{\mu}{2}$.
To determine if this point is actually a mass point, we recall  \cite[Proposition 2.3.12]{Rice}, which tells us that if there is a mass point at $x^*$, then its mass is
\[
\lim_{t\rightarrow0^+}(-it)m(x+it;\nu^*)
\]
Plugging in $x^*=\frac{-1-\mu^2/2}{\mu}$, then we consider
\begin{align*}
\nu^*\{x^*\}&=\lim_{t\rightarrow0^+}\frac{2(\frac{-1-\mu^2/2}{\mu}+it)+3\mu+2\sqrt{(\frac{-1-\mu^2/2}{\mu}+it-\frac{\mu}{2})^2-4}}{4\mu}\\
&=\frac{\mu-\frac{1}{\mu}+\sqrt{(\mu+\frac{1}{\mu})^2-4}}{2\mu}
\end{align*}
where the branch cut of $\sqrt{z^2-4}$ is taken on $[-2,2]$ so that it is positive if $z>2$.  With this choice of square root, one finds that $\nu^*\{x^*\}=0$ if and only if $|\mu|\leq1$, so $\nu^*$ only includes a mass point if $|\mu|>1$.
\end{proof}

\noindent\textit{Remark.}  We actually know from \cite[Chapter 2]{Rice} that the spectral measure of $\mcj^*(\mu)$ and the vector $e_0$ is given by
\[
\chi_{[-2+\frac{\mu}{2},2+\frac{\mu}{2}]}\frac{\sqrt{4-(x-\frac{\mu}{2})^2}}{2\pi(\mu x+\frac{\mu^2}{2}+1)}\,dx+\frac{\mu-\frac{1}{\mu}+\sqrt{(\mu+\frac{1}{\mu})^2-4}}{2\mu}\delta_{-\frac{\mu}{2}-\frac{1}{\mu}}
\]
assuming $|\mu|>1$.  If $|\mu|\leq1$, then the spectral measure is just the absolutely continuous part of this measure.

\medskip

\begin{figure}[h!]
  \centering
  \includegraphics[width=0.55\textwidth]{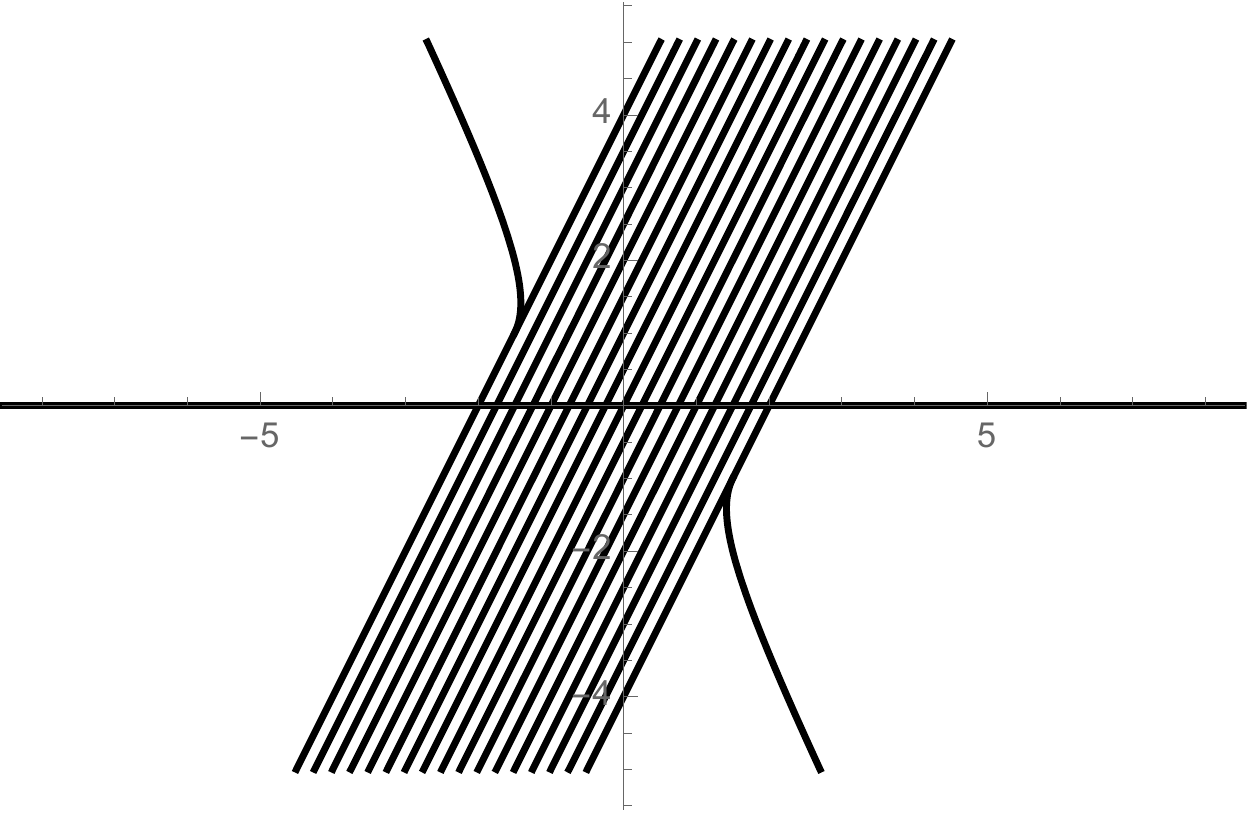}
\caption{The spectrum of $\mcj^*(\mu)$ when $\mu$ is real.  The vertical axis corresponds to $\mu$ and then the portion of the picture at that horizontal level is the spectrum of $\mcj^*(\mu)$.}
\label{fig:1}
\end{figure}



Proposition \ref{spectrum} identifies the spectrum of the linear pencil $\mcj_{free}+\mu D$, where $\mcj_{free}=\mathfrak{L}+\mathfrak{R}$ and $D=\mathrm{diag}(-1/2,1/2,1/2,1/2,\ldots)$.

\begin{proof}[Proof of Theorem \ref{t1}:]
(a)  Assume $-1\leq\mu\leq1$.  We have already seen that the spectrum of $M_{\mu}$ consists of the zeros of the polynomials $G_k(z,\mu)$ and the point $4-\mu$.  Proposition \ref{spectrum} and the formula \eqref{pgform} imply all of the zeros of $G_k(z,\mu)$ are inside $[-4-\mu,4-\mu]$.  Finally, the information about the zeros of orthogonal polynomials presented in Section \ref{op} shows that these zeros densely pack this interval.

(b) Assume $|\mu|>1$.  As in part (a), the zeros of $G_k(z,\mu)$ densely pack $[-4-\mu,4-\mu]$.  Also, the information about the zeros of orthogonal polynomials presented in Section \ref{op} shows that there is a collection of zeros of $P_{k,\mu}(z)$ (defined before Proposition \ref{spectrum}) approaching $-\mu/2-1/\mu$ as $k\rightarrow\infty$.  The relation \eqref{pgform} then implies there is a sequence of zeroes of $G_k(z,\mu)$ whose only accumulation point as $k\rightarrow\infty$ is $\mu+2/\mu$.  It is easy to see that $\mu+2/\mu\not\in[-4-\mu,4-\mu]$.
\end{proof}

\medskip

\begin{proof}[Proof of Theorem \ref{t2}:]
As in \cite[Section 7]{GZ01}, we know that the spectral measure $\nu_{\mu}$ is the weak-$*$ limit of the normalized counting measures of the zeros of $\Phi_n(\lambda,\mu)$.  The formula \eqref{phiform} then gives the desired formula for $\nu_{\mu}$.
\end{proof}

\section{Computation of the Spectral Measure}\label{specomp}

Theorem \ref{t2} provides a general formula for $\nu_{\mu}$, so now we want to extract additional information from it.  From our earlier observations, we know that when $\mu$ is real, the zeros of $G_k(\lambda,\mu)$ will be in the interval $(-4-\mu,4-\mu)$ when $|\mu|\leq1$ and when $|\mu|>1$, the zeros will mostly be in this interval, but there will be a single zero tending monotonically to $\mu+\frac{2}{\mu}$ as $k\rightarrow\infty$.

Our next task is to determine precisely when $G_k(\lambda,\mu)$ contains a zero outside of the interval $[-4-\mu,4-\mu]$.  We have already mentioned that this occurs when $|\mu|>1$ and $k$ is sufficiently large, but we want to be more precise about what we mean by ``sufficiently large."

To complete this calculation, let us assume that $\mu>1$ (the case $\mu<-1$ can be handled similarly).  Using our formula for $H_k=G_{k-1}$ and the fact that $U_k(1)=k+1$, we conclude that $G_k(4-\mu,\mu)=0$ if and only if $\mu=\frac{k+1}{k}$ and $G_k(\lambda,\mu)$ changes sign in the interval $(4-\mu,\infty)$ if $\mu>\frac{k+1}{k}$.  Thus, for a fixed $\mu>1$ we define $m(\mu)$ so that
\begin{equation}\label{mucut}
\frac{m(\mu)+1}{m(\mu)}\leq\mu<\frac{m(\mu)}{m(\mu)-1}
\end{equation}
and observe that at the points $\mu=2,\frac{3}{2},\frac{4}{3},\ldots$ there is a bifurcation of the multiplicity of the zeros of $\Phi_n(\lambda,\mu)$ that are outside the interval $[-4-\mu,4-\mu]$.  We summarize our conclusions in the following proposition.

\begin{proposition}\label{kzero}
If $\mu>1$, then $G_k(\lambda,\mu)$ has a zero outside the interval $(-4-\mu,4-\mu)$ if and only if $k\geq m(|\mu|)$.
\end{proposition}

\noindent\textit{Remark.}  Proposition \ref{kzero} explains how we found the value of $7/6$ in describing Figure 2 before Section \ref{emult}.  Of course the conclusion of Proposition \ref{kzero} also holds true for $\Phi_k(\lambda,\mu)$.

\smallskip

If $\frac{k+1}{k}<\mu<\frac{k}{k-1}$, then $G_n(z,\mu)$ has a zero outside of $[-4-\mu,4-\mu]$ for all $n\geq k$ and thus the same holds true for $\Phi_n(\lambda,\mu)$.  The multiplicity of any such zero is equal to the exponent on the corresponding factor $G_m$ in \eqref{phiform}.

As we have mentioned, the spectral measure that we wish to compute can be realized as the limiting measure of the counting measures of the zeros of $\Phi_n(\lambda,\mu)$.  
If $T_1$ and $T_2$ are bounded operators, we define the real joint spectrum of $T_1$ and $T_2$ to be all pairs of real numbers $(r_1,r_2)$ such that $T_1-r_1T_2-r_2I$ does not have a bounded inverse.  We can now completely describe the real joint spectrum of the linear pencil of operators that we have been considering.

\begin{theorem}\label{t3}
The real joint spectrum of operators of convolution in $\ell^2(\mcl)$ given by elements $a+a^{-1}+b+b^{-1}$ and $c$ from the group algebra consists of the strip $\{|\lambda+\mu|\leq4\}$, the line $\{\lambda=\mu\}$, and a countable collection of disjoint curves $\{\gamma_n^{\pm}\}_{n=1}^{\infty}$ that are asymptotic to the line $\{\lambda=\mu\}$.  Each curve $\gamma_n^{\pm}$ is outside the strip $\{|\lambda+\mu|\leq4\}$ and the curve $\gamma_n^+$ meets the strip $\{|\lambda+\mu|\leq4\}$ at the point $(\lambda=3-1/n,\mu=1+1/n)$ (similarly $\gamma_n^-$ meets the strip $\{|\lambda+\mu|\leq4\}$ at the point $(\lambda=-3+1/n,\mu=-1-1/n)$).
\end{theorem}

\begin{proof}
We have seen that the spectral measure $\nu_{\mu}$ is a pure point measure, so the spectrum is the closure of the set of eigenvalues.  The conclusion now follows from Theorem \ref{t1}.
\end{proof}

The application of Theorem \ref{t2} is not especially transparent because some of the sums may be redundant, meaning the zero sets of the polynomials $G_k$ may not be disjoint.  Thus, the next step is to determine the multiplicity of the zeros of $\Phi_n(\lambda,\mu)$ (again, here we are thinking of $\mu$ as a fixed constant and $\lambda$ is our variable).  We can interpret our results as an understanding of when the curves $\{(\lambda,\mu):G_n(\lambda,\mu)\}$ intersect for different values of $n$ and where these curves can intersect.  The main idea of our findings is summarized in our next three propositions.  Before we can state it, we need to define three sets:
\begin{align*}
\mcb_1&:=\left\{-\frac{\sin((n+1)t)}{\sin(nt)}: t=\frac{p}{q}\pi;\,p,q,n\in\bbN;\, p<q\right\}\\
\mcb_2&:=\{t:U_k(t/2)=0\,\mathrm{for\, some}\, k\in\bbN \}\\
\mcb_3&:=\left\{1+\frac{1}{k}:k\in\bbN \right\}
\end{align*}
The set $\mcb_3$ is motivated by Proposition \ref{kzero} and the inequalities \eqref{mucut}.  The orthogonality of the Chebyshev polynomials implies $\mcb_2\subseteq[-1,1]$.  Recall Niven's Theorem states that if $x/\pi$ and $\sin(x)$ are both rational, then $\sin x\in\{0,\pm1/2,\pm1\}$  (see \cite[Corollary 3.12]{Niven}).  Notice that \eqref{utrig} and Niven's Theorem together imply $\mcb_2\cap\mcb_3=\emptyset$.  The set $\mcb_1$ will tell us how to determine if a mass point in the support of $\nu_{\mu}$ appears in several terms of the sum defining $\nu_{\mu}$ in Theorem \ref{t2}.  We see its relevance in our next result, which concerns zeros inside the interval $[-4-\mu,4-\mu]$.  For convenience, we parametrize this interval by $\{-\mu-4\cos(t):0\leq t\leq\pi\}$.

\begin{proposition}\label{bees}
$\,$
\begin{itemize}
\item[(i)]  Suppose $G_{n}(-\mu-4\cos(t),\mu)=0$ for some $t\in(0,\pi)$ and $t$ is not a rational multiple of $\pi$.  Then $G_m(-\mu-4\cos(t),\mu)\neq0$ for all $m>n$.
\item[(ii)]  Suppose $G_{n}(-\mu-4\cos(t),\mu)=0$ for some $t\in(0,\pi)$ and $t$ is a rational multiple of $\pi$.  Then $\mu\in\mcb_1$ and if we write $t=\frac{p}{q}\pi$, then the set  of indices $k$ for which $G_k(\cdot,\mu)$ vanishes at $-\mu-4\cos\left(t\right)$ forms an arithmetic progression with step size $q$ and beginning at some $k_0\in\{1,\ldots,q\}$.
\end{itemize}
\end{proposition}

\begin{proof}
To prove (i) first recall the relations \eqref{hgk} and (\ref{hform}).  Notice that $G_{k+1}(-\mu-4\cos(t),\mu)=0$ is equivalent to the condition
\begin{equation}\label{sink}
\sin((k+1)t)=-\mu\sin(kt)
\end{equation}
If $k=0$, then this implies $t$ is an integer multiple of $\pi$, so our hypotheses in part (i) exclude this possibility.  If $k>0$, then \eqref{sink} can be rewritten as
\[
-\mu=\cos(t)+\sin(t)\cot(kt)
\]
If this identity holds true for multiple natural numbers (say $m_1<m_2$), then it must be the case that
\begin{equation}\label{trat}
t=\frac{M}{m_2-m_1}\pi
\end{equation}
for some $M\in\bbN$.  This implies $t\in\bbQ\pi$, which contradicts our assumption and proves (i).

\smallskip

To prove (ii) notice that the calculations from the proof of part (i) imply $\mu\in\mcb_1$.  Indeed, the calculations in that proof are all reversible, so we see that $G_{m_1+1}(-\mu-4\cos(t),\mu)=G_{m_2+1}(-\mu-4\cos(t),\mu)=0$ if and only if
\begin{align*}
t&=\frac{M\pi}{m_2-m_1}\\
\mu&=-\cos\left(\frac{M\pi}{m_2-m_1}\right)-\sin\left(\frac{M\pi}{m_2-m_1}\right)\cot\left(\frac{m_1M\pi}{m_2-m_1}\right)\\
&=-\cos\left(\frac{M\pi}{m_2-m_1}\right)-\sin\left(\frac{M\pi}{m_2-m_1}\right)\cot\left(\frac{m_2M\pi}{m_2-m_1}\right)
\end{align*}
for some $M\in\bbN$.  Thus, if $t=\frac{p}{q}\pi$ and $G_{n+1}(-\mu-4\cos(t),\mu)=0$, then
\begin{align*}
\mu&=-\cos\left(\frac{p\pi}{q}\right)-\sin\left(\frac{p\pi}{q}\right)\cot\left(\frac{np\pi}{q}\right)\\
&=-\cos\left(\frac{p\pi}{q}\right)-\sin\left(\frac{p\pi}{q}\right)\cot\left(\frac{(n+jq)p\pi}{q}\right),\qquad\qquad j\in\bbZ
\end{align*}
and hence $G_{n+jq+1}(-\mu-4\cos(t),\mu)=0$ for all $j\in\bbZ$ as long as $n+jq+1>0$.  This proves the claims made in (ii).
\end{proof}

Now we will see that as long as we avoid the sets $\{\mcb_j\}_{j=1}^3$, we need not worry about the complications presented in Proposition \ref{bees}.

\begin{proposition}\label{zerocount}
If $\mu\not\in\mcb_1\cup\mcb_2\cup\mcb_3$ and $\Phi_n(\lambda,\mu)=0$, then the multiplicity of $\lambda$ as a root of $\Phi_n(z,\mu)$ is $1$ if $\mu=4-\lambda$ or is equal to the exponent on $G_k$ in \eqref{phiform}, where $k$ is the unique natural number $m$ for which $G_m(\lambda,\mu)=0$.  In the latter case, the measure $\nu_{\mu}$ assigns that mass $2^{-(k+1)}$ to the point $\lambda$.
\end{proposition}

Proposition \ref{zerocount} tells us that generically we may apply Theorem \ref{t2} at face value.  Each of the pure point measures in the infinite sum defining $\nu_{\mu}$ have disjoint supports, so one can easily calculate the mass of any particular point to which $\nu_{\mu}$ assigns weight.  However, there are cases where these measures do not have disjoint supports (namely when $\mu\in\mcb_1\cup\mcb_2\cup\mcb_3$) and further analysis is required to determine the mass of any one particular point.



\begin{proof}[Proof of Proposition \ref{zerocount}:]
We saw in the proof of Proposition \ref{bees} that a the existence of a common root of $G_{n+1}(\cdot,\mu)$ and $G_{m+1}(\cdot,\mu)$ for $m,n>0$ requires $\mu\in\mcb_1$.  Also notice that $G_k(\mu,\mu)=0$ if and only if $\mu\in\mcb_2$.  Thus, if $\mu\not\in\mcb_2$, then we can extend our conclusion to the case when $m$ or $n$ is equal to $0$.  Finally, note that Proposition \ref{kzero} tells us that if $\mu\not\in\mcb_3$, then $G_k$ shares no common roots with the factor $(4-\lambda-\mu)$.  The desired conclusion is now apparent.
\end{proof}
 
We can now conclude that if $\mu\not\in\mcb_1\cup\mcb_2\cup\mcb_3$, $\Phi_n(\lambda,\mu)=0$, and $\lambda\in[-4-\mu,4-\mu]$, then the multiplicity of $\lambda$ as a root is
\[
\begin{cases}
1 \hspace{2in} & \mathrm{if}\qquad  \lambda=4-\mu\\
2^{n-1-j} & \mathrm{if} \qquad G_{j}(\lambda,\mu)=0,\qquad j=1,\ldots,n-1\\
1 & \mathrm{if} \qquad G_{n}(\lambda,\mu)=0
\end{cases}
\]
Thus, the spectral measure $\nu_{\mu}$ (from Theorem \ref{t2}) will assign weight $\frac{1}{2^{j+1}}$ to each zero of $G_j(\lambda,\mu)$.


Let us see what happens if we relax some of the assumptions from Proposition \ref{zerocount}.  

\begin{proposition}\label{relax}
$\,$
\begin{itemize}
\item[(i)] If $\mu\not\in\mcb_1$ but $\mu\in\mcb_3$ (which implies $\mu\not\in\mcb_2$ by our earlier remark) and $\Phi_n(\lambda,\mu)=0$, then the multiplicity of $\lambda\in[-4-\mu,4-\mu]$ as a root of $\Phi_n(\lambda,\mu)$ is
\[
\begin{cases}
1+2^{n-1-k} \hspace{1.2in} & \mathrm{if}\qquad  \lambda=4-\mu\\
2^{n-1-j} & \mathrm{if} \qquad G_{j}(\lambda,\mu)=0,\qquad j=1,\ldots,n-1,\quad\lambda\neq4-\mu\\
1 & \mathrm{if} \qquad G_{n}(\lambda,\mu)=0,\qquad\lambda\neq4-\mu
\end{cases}
\]
as long as $n>k$, where $k$ is chosen so that $\mu=1+k^{-1}$.  The spectral measure $\nu_{\mu}$ assigns weight $2^{-(1+k)}$ to the point $4-\mu$ and weight $2^{-(j+1)}$ to each zero of $G_j$ not equal to $4-\mu$.
\item[(ii)] Suppose $\mu\not\in\mcb_1$ but $\mu\in\mcb_2$ (which implies $\mu\not\in\mcb_3$ by our earlier remark) and $\Phi_n(\lambda,\mu)=0$.  We will suppose $U_k(-\mu/2)=0$ for some $k\in\bbN$ and we assume $k$ is the smallest index for which this holds.  The multiplicity of $\lambda\in[-4-\mu,4-\mu]$ as a root of $\Phi_n$ is (where $J=\left\lfloor\frac{n-1}{k+1}\right\rfloor$)
\[
\begin{cases}
1 \hspace{1.2in} & \mathrm{if}\qquad  \lambda=4-\mu\\
2^{n-2}+\frac{(2^{(k+1)J}-1)2^{n-(k+1)J}}{4(2^{k+1}-1)} & \mathrm{if} \qquad \lambda=\mu\\
2^{n-1-j} & \mathrm{if} \qquad G_{j}(\lambda,\mu)=0,\qquad j=2,\ldots,n-1,\quad\lambda\neq\mu\\
1 & \mathrm{if} \qquad G_{n}(\lambda,\mu)=0
\end{cases}
\]
\end{itemize}
where we assume $n$ is such that $G_n(\mu,\mu)\neq0$.  The spectral measure $\nu_{\mu}$ assigns weight $2^{-(j+1)}$ to each zero of $G_j$ not equal to $\mu$ and also satisfies
\[
\nu_{\mu}(\{\mu\})=\frac{1}{4}+\frac{1}{4(2^{k+1}-1)}.
\]
\end{proposition}

\begin{proof}
For part (i), we notice that $G_k(4-\mu,\mu)=0$ by Proposition \ref{kzero} so we have simply added the multiplicities of the factors of $(4-\lambda-\mu)$ and $G_k$ to get the multiplicity of $4-\mu$ as a root.

For part (ii), notice that the given information and \eqref{utrig} imply $\mu/2=\cos(j\pi/(k+1))$ for some $j\in\{1,\ldots,k\}$.  Then $U_m(-\mu/2)=0$ for all $m$ of the form $m=qk+q-1$ with $q\in\bbN$.  This gives us an arithmetic progression of indices for which $\mu$ is a root.  Summing up the corresponding multiplicities gives the desired result.
\end{proof}

The calculation of the multiplicities of roots is more complicated when $\mu\in\mcb_1$.  Since the set of indices at which a zero is repeated forms an arithmetic progression, finding the weight that $\nu_{\mu}$ assigns to that point requires a calculation similar to that presented in Proposition \ref{relax}ii.  However, there is the added difficulty of determining the first index in this set, which prevents us from presenting a concise formula.  The following list summarizes what we have so far learned about $\nu_{\mu}$:

\begin{itemize}
\item  If $\mu\not\in\mcb_1\cup\mcb_2\cup\mcb_3$, then to each mass point in $\lambda\in\supp(\nu_{\mu})$ we can associate to $\lambda$ a unique index $j$ such that $G_j(\lambda,\mu)=0$.  It then holds that $\nu_{\mu}\{\lambda\})=2^{-(j+1)}$.
\item  If $\mu\in\mcb_3\setminus(\mcb_1\cup\mcb_2)$ and $\mu=1+k^{-1}$, then $\mu_{\nu}(\{4-\mu\})=2^{-(k+1)}$.  If $\lambda\in\supp(\nu_{\mu})$ is a mass point of $\nu_{\mu}$ and $\lambda\neq4-\mu$, then we can associate to $\lambda$ a unique index $j$ such that $G_j(\lambda,\mu)=0$.  It then holds that $\nu_{\mu}\{\lambda\})=2^{-(j+1)}$.
\item  If $\mu\in\mcb_2\setminus(\mcb_1\cup\mcb_3)$ and $U_k(-\mu/2)=0$ and $k$ is the smallest such index, then $\mu_{\nu}(\{\mu\})=1/4+1/(4(2^{k+1}-1))$.  If $\lambda\in\supp(\nu_{\mu})$ is a mass point of $\nu_{\mu}$ and $\lambda\neq\mu$, then we can associate to $\lambda$ a unique index $j$ such that $G_j(\lambda,\mu)=0$.  It then holds that $\nu_{\mu}\{\lambda\})=2^{-(j+1)}$.
\item  If $\mu\in\mcb_1\setminus(\mcb_2\cup\mcb_3)$, then $\mu=-\cos(t)-\sin(t)\cot(nt)$ for some natural number $n$ and some $t=\frac{p}{q}\pi\in(0,\pi)$.  If $n$ is the smallest such natural number for which this is true, then $\nu_{\mu}(\{-\mu-4\cos(t)\})=\frac{2^q}{2^{n+1}(2^q-1)}$.
\end{itemize}

\subsection*{Example.}  Consider the case when $\mu=0$ and let us calculate the multiplicity of the zero $(0,0)$ of $\Phi_n$.  In this case
\[
\Phi_n(\lambda,0)=(4-\lambda)\left(2U_1\left(\frac{-\lambda}{4}\right)\right)^{2^{n-2}}\cdots\left(2^{n-1}U_{n-1}\left(\frac{-\lambda}{4}\right)\right)\left(2^{n}U_{n}\left(\frac{-\lambda}{4}\right)\right)
\]
Notice that $0$ is a root of $U_k$ if and only if $k$ is odd, so the multiplicity of the root at $0$ is
\[
\begin{cases}
2^{n-2}+2^{n-4}+\cdots+2^3+2+1\qquad\qquad & \mathrm{if}\quad $n$\,\mathrm{is\, odd}\\
2^{n-2}+2^{n-4}+\cdots+2^4+2^2+1 & \mathrm{if}\quad $n$\,\mathrm{is\, even}
\end{cases}
\]
This agrees with our earlier observations because $\lambda=0=\cos(\pi/2)$, so every collection of two consecutive natural numbers contains an index $k$ for which $G_k(0,0)=0$.  By normalizing and sending $\nri$, we see that $\nu_{0}(\{0\})=1/3$.  This is a special case of a result from \cite{GZ01}, which shows $\nu_0(\{\cos(p\pi/q)\})=1/(2^q-1)$.  This can also be obtained from our above discussion, since $0\in\mcb_1\setminus(\mcb_2\cup\mcb_3)$ and we can apply the given formula with $n=1$ and $t=\pi/2$.

\subsection*{Example.}  The case $\mu=2$ also presents an interesting example because in this case $\mu=4-\mu$ and $\mu=\frac{1+1}{1}$ and several of the above special cases apply when calculating the multiplicity of the root $\Phi_n(2,2)$.  In this case
\begin{align*}
&\Phi_n(\lambda,2)=(2-\lambda)\left(2U_1\left(\frac{-\lambda-2}{4}\right)+4\right)^{2^{n-2}}\left(4U_2\left(\frac{-\lambda-2}{4}\right)+8U_1\left(\frac{-\lambda-2}{4}\right)\right)^{2^{n-3}}\\
&\,\times\cdots\times
\left(2^{n}U_{n}\left(\frac{-\lambda-2}{4}\right)+2^{n+1}U_{n-1}\left(\frac{-\lambda-2}{4}\right)\right)
\end{align*}
Notice that $\lambda=2$ is a root due to the first factor of $(2-\lambda)$.  Also, $\lambda=2$ will be a root for any $U_k$ such that $U_{k}(-1)+2U_{k-1}(-1)=0$.  This is precisely when $k=1$, so $2$ is a root of multiplicity $1+2^{n-2}$.

\subsection{The set $\mcb_1$.}

The set $\mcb_1$ is the set of all real numbers $r$ that can be written as
\[
r=\frac{\sin((n+1)q\pi)}{\sin(nq\pi)}
\]
for some $q\in\bbQ$ and $n\in\bbN$.

\begin{proposition}\label{b1}
The set $\mcb_1$ is dense in $\bbR$.
\end{proposition}

\begin{proof}
It is easy to see that the closure of $\mcb_1$ includes all numbers of the form
\[
r=\frac{\sin((n+1)x)}{\sin(nx)}
\]
for some $x\in\bbR$ and $n\in\bbN$.

Now, pick any $r_0\in\bbR$.  It is clear the we can find points $z_1=(x_1,\sin(x_1))$ and $z_2=(x_2,\sin(x_2))$ on the graph of $\sin(x)$ such that $r_0=\sin(x_2)/\sin(x_1)$ and $x_2>x_1$.  In fact, there are infinitely many pairs of such points with the same distance between them due to the periodicity of the graph of $\sin(x)$.  If $s:=x_2-x_1$ is an irrational multiple of $\pi$, then consider the set
\[
\mco_s:=\left\{\frac{\sin((n+2)s)}{\sin((n+1)s)}\right\}_{n\in\bbN}
\]
The numbers $\{(n+1)s\,\mathrm{mod}\,2\pi\}_{n\in\bbN}$ are dense in $[0,2\pi]$ and hence by taking $\nri$ through the appropriate subsequence, a subsequence in $\mco_s$ converges to $\sin(x_2)/\sin(x_1)=r_0$ as desired.

If $s$ is a rational multiple of $\pi$, then we can choose $x_3$ arbitrarily close to $x_2$ so that $x_3-x_1$ is an irrational multiple of $\pi$.  Our above reasoning shows that $\sin(x_3)/\sin(x_1)$ is in $\mcb_1$ and hence by a limiting argument, $\sin(x_2)/\sin(x_1)=r_0$ is in $\mcb_1$ also.
\end{proof}

\subsection{Bulk Distribution of Zeros}\label{bulkz}

Let us now focus our attention on the distribution of the zeros of $\Phi_n(\lambda,\mu)$ for a fixed $\mu$ and large $n$.  Suppose $\lambda_1$ is such that $U_{k-1}\left(\frac{-\lambda_1-\mu}{4}\right)=0$ and $\lambda_2>\lambda_1$ is the next real number so that $U_{k-1}\left(\frac{-\lambda_2-\mu}{4}\right)=0$.  By the interlacing property of zeros of Chebyshev polynomials, it is true that $U_{k}\left(\frac{-\lambda_1-\mu}{4}\right)$ and $U_{k}\left(\frac{-\lambda_2-\mu}{4}\right)$ are both non-zero and have opposite sign.  Therefore, in moving from $\lambda_1$ to $\lambda_2$, the fraction $U_{k}\left(\frac{-\lambda-\mu}{4}\right)/U_{k-1}\left(\frac{-\lambda-\mu}{4}\right)$ takes all values in $\bbR$.  In particular, there is a $\lambda^*$ where $U_{k}\left(\frac{-\lambda^*-\mu}{4}\right)/U_{k-1}\left(\frac{-\lambda^*-\mu}{4}\right)=\mu$ and this is a zero of $H_{k+1}(\lambda,\mu)=G_{k}(\lambda,\mu)$.  Thus, the zeros of $G_k(\lambda,\mu)$ satisfy the same law as the zeros of $U_k$, but shifted by $\mu$ and rescaled by a factor of $4$.

Note that the above procedure only accounts for $k-2$ zeros of $G_k(\lambda,\mu)$, which has degree $k$.  To find the remaining two zeros, we apply the above reasoning with $\lambda_1$ replaced by $-\infty$ or $\lambda_2$ replaced by $\infty$.  Using the fact that we know how $U_k$ behaves at the endpoints of the support of its measure of orthogonality, we can see that when $|\mu|>1$, we will have one zero of $G_k$ predictably outside this (translated and rescaled) interval for all large $k$.  Using the ratio asymptotic formulas for the Chebyshev polynomials, we can see that this zero approaches $\mu+2/\mu$ as $k\rightarrow\infty$.

\section{The Relation with the Theory of Random Jacobi $\&$ Schr\"{o}dinger Operators}\label{schro}

Surprisingly, there is an interesting connection between the above ideas and the theory of random operators acting on $\ell^2(\bbZ)$.  This was first observed by L. Grabowski and B. Virag in \cite{G14,GV,KV17} who used this connection to show that the spectral measure associated with the operator $H_{\beta}$ of convolution with the element $a+a^{-1}+\beta c\in\bbR[\mcl]$ is singular continuous.  This was achieved through a modification of a method from \cite{Austin13}, which appeared in an unpublished work of Grabowski and Virag \cite{GV} and for which an exposition appears in the appendix of \cite{G15}.  The idea is to use an application of the Fourier transform using Pontryagin's dual $\hat{A}=\prod\bbZ/2\bbZ$ of the base $A=\bigoplus_{\bbZ}\bbZ/2\bbZ$ of the wreath product
\[
\mcl=\bbZ/2\bbZ\wr\bbZ=\left(\bigoplus_{\bbZ}\bbZ/2\bbZ\right)\rtimes\bbZ
\]
which transforms the operator $H_{\beta}$ into a random Schr\"{o}dinger operator $\tilde{H}_{\beta}$ acting on $\ell^2(\bbZ)$ by
\[
(\tilde{H}_{\beta,\omega}u)(n)=u(n-1)+V_{\omega}(n)u(n)+u(n+1),
\]
where $\{u(n)\}_{n\in\bbZ}\in\ell^2(\bbZ)$ and $\omega$ is a sequence of i.i.d random variables taking values in $\{0,1\}$ with Bernoulli distribution the assigns probability $1/2$ to $0$ and $1$.  The potential $V_{\omega}(n)$ is then a random potential that takes the value $\beta$ in the $n^{th}$ coordinate if the $n^{th}$ coordinate of $\omega$ is equal to $1$ and takes the value $0$ otherwise.  The surprising property of this correspondence is that the spectral measure of $H_{\beta}$ associated with the delta function $\delta_1\in\ell^2(\mcl)$ coincides with the density of states of the family $\{\tilde{H}_{\beta,\omega}\}_{\omega\in\Omega=\{0,1\}^{\bbZ}}$, which is the average over $\Omega$ of the spectral measures of $\tilde{H}_{\beta,\omega}$ associated with the delta function at the origin $\delta_0\in\ell^2(\bbZ)$.  This fact was discovered by L. Grabowski and is valid in much greater generality \cite{G14}.

Using results from \cite{MM} 
Grabowski and Virag concluded that $H_{\beta}$ has a singular continuous spectral measure when $\beta$ is sufficiently large.  This is the first example of a group and convolution operator on it with singular continuous spectral measure.

Our result allows one to reason in the opposite direction.  Using information about the spectrum and spectral measure of Markov operators $M_{\mu}$ of convolution on $\mcl$, we are able to deduce new information about the spectrum and the integrated density of states for dynamically defined random Jacobi operators, where the randomness is ruled by a uniform Bernoulli measure on the space $\{0,1\}^{\bbZ}$.  Recall that the classical Anderson model concerns the Schr\"{o}dinger operator $H_{\omega}$ acting on $\ell^2(\bbZ)$ by
\[
(H_{\omega}u)(n)=u(n-1)+V_{\omega}(n)u(n)+u(n+1),
\]
where $V_{\omega}(n)=g(T^n\omega)$, $T$ is the left shift in $\Omega=[a,b]^{\bbZ}$ with $a<b$, $a,b\in\bbR$, and $g:\Omega\rightarrow\bbR$ is the sampling function which is given by evaluation at the $0$ coordinate.  The randomness is described by the product measure $\nu=\rho^{\bbZ}$, where $\rho$ is a probability measure on $[a,b]$.  It is well-known that the spectrum of $H_{\omega}$ as a set is almost surely independent of $\omega$ and is equal to $[-2,2]+\supp(\rho)$.  Thus as a set, $\spm(H_{\omega})$ is equal to a union of finitely many intervals.  Applying this to the case considered by Grabowski in \cite{G14}, we find $\spm(\tilde{H}_{\beta,\omega})=\spm(H_{\beta})=[-2,2]\cup[-2+\beta,2+\beta]$.

We now consider the related operator acting on $\ell^2(\bbZ)$ given by \eqref{anderson}, where $f,g:[a,b]^{\bbZ}\rightarrow\bbR$, $\omega$ is as above, and the randomness is ruled as before by a probability measure $\nu=\rho^{\bbZ}$ where $\rho$ is a probability measure on $[a,b]$.  Thus, the operator in matrix form has the structure of a Jacobi matrix with entries determined by the functions $f$ and $g$.  It is quite probable that many results from the classical theory of dynamically defined Jacobi matrices associated with Schr\"{o}dinger operators can be translated to this more general setting.  Such a generalization in the case of minimal subshifts over a finite alphabet is done by Beckus and Pogorzelski in \cite{BP}.

Consider the operator $H_{\mu}$ obtained by applying the Fourier transform given by the Pontryagin duality between $\hat{A}$ and $A$ as described in \cite{G14,G15,G16,KV17} to the operator $M_{\mu}$ of convolution by the element $a+a^{-1}+b+b^{-1}-\mu c\in\bbR[\mcl]$ in $\ell^2(\bbZ)$.  This operator $H_{\mu}$ is of the form
\[
\int_{\Omega}H_{\mu,\omega}d\nu(\omega)
\]
in the Hilbert space
\[
\int_{\Omega}\ell^2(\Gamma_{\omega})d\nu(\omega)
\]
where $\Omega=[0,1]^{\bbZ}=\hat{A}$ serves as the dual of $A$ and $\Gamma_{\omega}$ is the $\bbZ$-orbit of $\omega$ for $\omega\in\Omega$ under the adjoint action of the active group $\bbZ$ of the wreath product $\bbZ/2\bbZ\wr\bbZ$ on $\hat{A}$ (which happens to be the action by powers of the shift).  If we identify $\Gamma_{\omega}$ with $\bbZ$ (for non-periodic $\omega$), then the generator of $\bbZ$ acts on $\Omega$ as the shift $T:(T\omega)_n=\omega_{n+1}$.
The operator $H_{\mu,\omega}$ is an operator of the type described in \eqref{anderson} with functions $f$ and $g$ defined on $\{0,1\}^{\bbZ}$ as
\[
f(\omega)=1+(-1)^{\omega_{-1}},\qquad\qquad g(\omega)=\mu(-1)^{\omega_0+1},
\]
where $\omega_n$ denotes the $n^{th}$ entry of $\omega\in\Omega$ and $\nu=\rho^{\bbZ}$, where $\rho$ is the uniform measure on the two-point set $\{0,1\}$.

Putting all the above together and keeping in mind that the matrices $H_{\mu,\omega}$ almost surely have a block diagonal structure with finite blocks, we come to the conclusion of Theorem \ref{ids}.

Theorem \ref{ids} shows that the nature of the density of states for random Jacobi matrices in the Anderson model can be very different from the case of random Schr\"{o}dinger operators.  Furthermore, along with the result in \cite{GZ01}, this provides perhaps the first nontrivial example of the exact computation of the integrated density of states for randomly defined Jacobi operators.

\section{Concluding Remarks and Open Problems}\label{end}

\subsection{The Resolvent Set}\label{reset}

Although Theorems \ref{t1} and \ref{t2} concern real values of $\mu$, one can consider complex values of $\mu$.
One can use similar calculations to those presented above to show that as $\mu$ varies through $\bbC$, the spectrum of $\mcj^*(\mu)$ is an interval or perhaps an interval and an isolated point.  The isolated point appears when $|\mu|>1$. 
We have also seen that the union of all the resolvent sets has four connected components in $\bbR^2$.


Now consider the resolvent set when $\mu$ is complex.  For a fixed $\mu$ the spectrum of $\mcj^*(\mu)$ is an interval and perhaps one isolated point.  For ease of visualization, we will consider the matrix $\mcj_1^*(\mu)$ (defined in the proof of Proposition \ref{spectrum}) instead of $\mcj^*(\mu)$.  In this case, the union of all the spectra is $\bbC\times[-2,2]\cup\mcs$, where $\mcs$ is the surface that is the image of $\{z:|z|>1\}$ under the map $z\rightarrow-z-1/z$, which is a bijection from $\{z:|z|>1\}$ to $\bbC\cup\{\infty\}\setminus[-2,2]$.  Thus, the union of all the spectra is the union of a three dimensional manifold with boundary (namely $\bbC\times[-2,2]$) and a two dimensional manifold (namely $\mcs$) and they are glued together in such a way that if $z_0\in\bbC$ has modulus $1$, then a boundary point of $\mcs$ attaches to the $z_0$ copy of $[-2,2]$ at the point $-z_0-1/z_0$.

It remains an open question to determine the joint spectrum of the pencil of operators $M_{\mu}-\lambda I$ acting on $\ell^2(\mcl)$ when $\mu,\lambda\in\bbC$.  The results of \cite{DG} and essential freeness of the action of $\mcl$ on $(\partial T,\nu)$, where $\nu$ is a uniform Bernoulli measure on $\partial T=\{0,1\}^{\bbN}$ show that for each $\mu\in\bbC$, the spectrum of $M_{\mu}$ coincides with the spectrum of the corresponding operator $\pi(m_{\mu})=\pi(a+a^{-1}+b+b^{-1}-\mu c)$ acting in $L^2(\partial T,\nu)$, where $\pi$ is the Koopman representation associated with the dynamical system $(\mcl,\partial T,\nu)$.  Our results describe the real joint spectrum (given by Figure 2).  The representation $\pi$ is a direct sum of finite dimensional subrepresentations, which implies that in the real situation (when $\mu\in\bbR$), it holds that
\begin{equation}\label{spu}
\spm\pi(m_{\mu})=\overline{\bigcup_{n=0}^{\infty}\spm\pi_n(m_{\mu})},
\end{equation}
where $\pi_n$, $n=0,1,2,\ldots$ are permutational representations given by the action of $\mcl$ on the level $n$ of the tree \cite{BG00,Gr11}.  If \eqref{spu} remains valid for complex values of $\mu$, then the part (a) of Theorem \ref{t1} remains true for $\mu\in\bbC$ and hence the joint spectrum of $M_{\mu}-\lambda I$ (or of $\pi(m_{\mu}-\lambda 1)$) is the closure of the union of curves $\{G_k(\lambda,\mu)=0\}$ for $k=0,1,2,\ldots$.

The validity of \eqref{spu} reduces to the question of whether or not the spectrum of the block diagonal matrix corresponding to the operator $\pi(m_{\mu})$ (whose blocks are finite matrices) is the closure of the union of the spectra of the blocks.  In the self-adjoint case, this is known to be true, but in the general case it need not hold.  Also, \eqref{spu} would hold if the matrix of $\pi(m_{\mu})$ is diagonalizable, but as our computations show, the eigenvalues of these matrices have large multiplicities (see also \cite{GZ01}), which complicates the issue of diagonalization.  We thus arrive at the following questions:

\begin{question}\label{q1}
Are the matrices $M_n(\mu)$ from Section \ref{spec} diagonalizable for all $\mu\in\bbC$?
\end{question}

\begin{question}\label{q2}
What is the joint spectrum of the pencil $M_{\mu}-\lambda I$ acting in $\ell^2(\mcl)$ (or equivalently, what is the joint spectrum of the pencil $\pi(m_{\mu}-\lambda 1)$ acting in $L^2(\partial T,\nu)$)?
\end{question}


Let $\{H_{\omega}\}_{\omega\in\Omega}$ be a random ergodic family of Jacobi operators on $\bbZ$ of the form \eqref{anderson} above.







An interesting question that is related to this discussion is about the decomposition of the representation $\pi_n$ into irreducible components.  If we know such a decomposition $\pi_n=\oplus_{i=1}^{\kappa_n}\pi_{n,i}$ with each $\pi_{n,i}$ irreducible, then for each element $m\in\bbC[G]$, the determinant $\det(\pi_n(m)-\lambda I_{2^n})$ (which is $\Phi_n(\lambda,\mu)$ in the notation of Section \ref{spec} with $m=m_{\mu}$) splits into the product $\prod_{i=1}^{\kappa_n}\det(\pi_{n,i}(m)-\lambda I_{2^n})$ and this reduces the problem to that of finding eigenvalues for the irreducible parts.  The high multiplicities of the eigenvalues suggests that perhaps the irreducible subrepresentations $\pi_{n,i}$ appear in the decomposition with high multiplicities.  However, the calculations for levels $n=1,2,\ldots,6,7$ show that this is not the case and that at least for these values of $n$, the pairs $(\mcl_n,\mathrm{st}_{\mcl_n}(1^n))$ are Gelfand Pairs, where $\mcl_n=\mcl/\mathrm{st}_{\mcl}(n)$, $\mathrm{st}_{\mcl}(n)$ is the stabilizer in $\mcl$ of level $n$, $\mathrm{st}_{\mcl_n}(1^n)$ is the stabilizer of the vertex $1^n=\underbrace{1\cdots1}_n$ for the natural action of the quotient group $\mcl_n$ on level $n$.

\begin{question}\label{q3}
For each $n$, what is the decomposition of $\pi_n$ into irreducible subrepresentations?
\end{question}

\subsection{Spectra and Spectral Measures}\label{qspec}

It is a widely open question as to what could be the shape of the spectrum $\spm(M_m)$ of convolution operators in $\ell^2(\mcl)$ given by elements $m\in\bbC[\mcl]$ of the group algebra.  Theorem \ref{t1} shows that in addition to the interval and union of two intervals, we can get an interval and a countable set of points accumulating to a point outside the interval.  As for the spectral measure $\mu_m$ (in the case of self-adjoint $m$) it can be a pure point measure (as shown in \cite{GZ01} and in this paper) or it can be a singular continuous measure (as in the case when $m=a+a^{-1}+\beta c$ with $\beta$ sufficiently large, where the spectrum is a union of two intervals).  Observe that all of this information is achievable already using elements $m$ of the form $\sum_{i=1}^kp_i(s_s+s_i^{-1})$, where $S=\{s_1,\ldots,s_k\}$ is a generating set of $\mcl$ and $p_i$ is a symmetric probabilistic distribution of $S\cup S^{-1}$ (thus corresponding to a Markov operator of a random walk on $\mcl$ given by this distribution).

\begin{question}\label{q4}
\begin{itemize}
\item[a)] What are all the shapes of the spectrum of operators of convolution $M_m$ in $\ell^2(\mcl)$ given by elements in $m\in\bbC[\mcl]$ of the group algebra?
\item[b)]  What are all the possible types of spectral measures  for operators as in part (a) with $m$ self-adjoint?
\end{itemize}
\end{question}

Let us elaborate on part (b) of Question \ref{q4}.  Any probability measure $\mu$ on $\bbR$ has a decomposition $\mu=\mu_{ac}+\mu_{sc}+\mu_{pp}$ into its absolutely continuous, singular continuous, and pure point parts.  Any of these parts may happen to be empty so, in principle, we may expect to have $8$ different possibilities for the decomposition with the nontrivial components.  From the above information, we know that $\mu=\mu_{sc}$ and $\mu=\mu_{pp}$ are possible.  What about the other $6$ cases?

Now let us restrict our attention to the operators $M_{m(\beta)}$ where $m(\beta)=a+a^{-1}+\beta c$.  Using the Fourier transform (as was explained above), an operator of this form transforms into a Schr\"{o}dinger operator with Bernoulli-Anderson random potential whose density of states $\mu_{m(\beta)}$ is continuous for all $\beta\in\bbR$, and is singular continuous when $\beta$ is sufficiently large \cite{MM}.

\begin{question}\label{q5}
For what values of $0\neq\beta\in\bbR$ does the density of states $\mu_{m(\beta)}$ have nontrivial absolutely continuous part?  What are the possibilities for $\mu_{m(\beta)}$ in terms of its absolutely continuous part and singular continuous part?
\end{question}


The values at zero of the spectral measure $\mu_m$ corresponding to operators of convolution $M_m$ in $\ell^2(G)$ given by elements $m\in\bbQ[G]$ are called ``$\ell^2$-Betti numbers arising \textit{from} the group $G$."  For a recursively presentable group $G$ (such as $\mcl$) they appear also as $L^2$-Betti numbers of closed manifolds \cite{GLSZ,Luck}.  Part (a) of the next question is \cite[Question 1.7]{G14}

\begin{question}\label{q6}
\begin{itemize}
\item[a)] What is the set of values of $\mu_m(0)$ for self-adjoint $m\in\bbQ[\mcl]$?
\item[b)] What is the set of values of $\mu_m(0)$ for self-adjoint $m\in\bbC[\mcl]$?
\item[c)] Do the sets from parts (a) and (b) coincide?
\end{itemize}
\end{question}

\begin{question}\label{q6.5}
For which rational values of $\mu$ does the operator $M_{\mu}$ given by $M_{\mu}=a+a^{-1}+b+b^{-1}-\mu1$ have a rational eigenvalues different from $\mu$?
\end{question}

The correspondence between convolution operators in $\ell^2[\mcl]$ and random generalized operators of Jacobi-Schr\"{o}dinger type should lead to interplay between methods of the theory of random operators and methods used in group theory, the theory of random walks on groups, and the theory of self-similar groups defined by finite automata.

The random operators corresponding to operators of convolution in $\ell^2(\mcl)$ are presented by random banded matrices of Jacobi type whose non-zero entries are of the form $f_i(T^k\omega)$ for $i=1,\ldots,k$ where $k$ is the number of upper diagonals, including the main diagonal, $\omega\in\Omega=\{0,1\}^{\bbZ}$, $T$ is a shift in $\Omega$, and $f_i$ is a function from $\Omega\rightarrow\bbR$ corresponding to the $i^{th}$ diagonal.  The pencil $M(\beta,\mu)=a+a^{-1}+\beta(b+b^{-1})+\mu c$, after an application of the Fourier transform, becomes the random Jacobi matrix ruled by two functions $f$ (on the off-diagonal)  and $g$ (on the diagonal).  Specifically,
\[
f(\omega)=1+\beta(-1)^{\omega_{-n}},\qquad\qquad
g(\omega)=\begin{cases}
\mu\qquad &\mathrm{if}\quad\omega_0=1\\
0 & \mathrm{otherwise}
\end{cases}
\]

\begin{question}\label{q7}
What is the joint spectrum of the pencil $M(\beta,\mu)-\lambda I$ when $\mu,\beta,\lambda\in\bbR$ and what is the corresponding spectral measure?
\end{question}

If $\beta\neq1$, then the off-diagonal function $f$ does not take the value $0$ and hence by the same argument as in \cite{DS84}, we know that the density of states is a continuous measure.  If $\beta=1$, then $f$ takes the value $0$ with positive probability and the argument from \cite{DS84} fails.  In fact, as we see from the results of \cite{GZ01} and this paper, the situation can be completely different from the $\beta\neq1$ case and the density of states may not only acquire mass points, but it could in fact become a pure point measure.

This matter is briefly discussed near the conclusion of \cite{DS84}, where it is written ``... the coefficients $J(x,y)$, $|x-y|=1$ could be zero with a non-zero probability, in which case the density of states would be discontinuous."  There is no argument provided in \cite{DS84} for this claim and it would be useful to clarify it.  If the word ``would" is replaced by the word ``could," then this paper and \cite{GZ01} confirm the claim.

The general case of random Jacobi-Schr\"{o}dinger matrices given by data $\{f_i(\omega)\}_{i=1}^k$ as described above looks to be too hard to handle.  The case of random Jacobi-Schr\"{o}dinger operators determined by two functions $f_1=g$ and $f_2=f$ and a Bernoulli system $(\Omega,T,\nu)$ seems to much more tractable and most of the known results of the classical theory should be valid for random Jacobi-Schr\"{o}dinger operators.  Some steps in this direction are done for instance in \cite{BP,BL85}.  The theorems of Pastur \cite{P80} about existence for the ergodic family $\{H_{\omega}\}_{\omega\in\Omega}$ of a set $\Sigma\subseteq\bbR$ such that $\spm(H_{\omega})=\Sigma$ extends to ergodic random Jacobi-Schr\"{o}dinger operators without extra efforts.  

\begin{question}\label{q8}
Let $\mcj_{\omega}(f,g)$ be a dynamically defined Jacobi matrix with diagonal entries determined by $g(T^n\omega)$ and off-diagonal entries determined by $f(T^n\omega)$.  Assume that $f(\omega)=0$ on a subset $\Omega_0\subset\Omega$ of positive measure and assume $\Omega\setminus\Omega_0$ has positive measure as well.  Will the density of states be a pure point measure?  Will it always have a non-trivial discrete component? Here $(T,\Omega,\nu)$ is a Bernoulli system (Anderson model).
\end{question}


\subsection{Novikov-Shubin invariants and Lifshits Tails}\label{nsinvar}

Given a self-adjoint $m\in\bbC[G]$, the Novikov-Shubin invariant of $m$ is defined as
\begin{equation}\label{nsdef}
\alpha(m)=\liminf_{\lambda\rightarrow0^+}\frac{\log(\mu_T(0,\lambda])}{\log\lambda}.
\end{equation}
Novikov-Shubin invariants were introduced for manifolds and then translated to the group case.  If $m\in\bbQ[G]$ and the group is recurseivel presentable, then $\alpha(m)$ can be realized as a fourth Novikov-Shubin invariant of some finite CW complex \cite{Luck}.  The papers \cite{G16,GZ04} show that already the lamplighter group $\mcl$ is rich enough to produce elements of $m\in\bbZ[G]$ with non-integer (or even irrational) $\ell^2$-Betti numbers (and with irrational Novikov-Shubin invariant).  In fact, the study of Novikov-Shubin invariants is a particular case of the study of regularity properties of the spectral distribution function $N_m(x):=\mu_m(-\infty,x]$.  For each point $E\in\spm(M_m)$ of the convolution operator $M_m$ given by $m\in\bbC[G]$, one can study local behavior of $N_m(x)-E_-$ on the left of $E$ (where $E_-=\lim_{x\rightarrow E_-}N(x)$ and local behavior of $N_m(x)-\mu_m(\{E\})$ to the right of $E$.  If on the right, one finds that $N_m(E+x)-\mu_m(\{E\})\sim x^{\alpha}$ for some $\alpha\in(0,\infty)$, then $\alpha$ is the Novikov-Shubin invariant of $m-\mu_m(\{E\})$.  But, the local behavior could be of type $-1/\log(x)$ or $e^{-1/x}$ or an even more exotic type, which corresponds to $\alpha=0$ or $\alpha=\infty$ respectively.  For different groups, elements $m\in\bbC[G]$ and different points of the spectrum dilatation behavior of $N_m(x)$ around $E\in\spm(M_m)$ can be very different.

Suppose $E_0$ and $E_+$ are extreme points of the spectrum of a convolution operator $M_m$ given by a self-adjoint element $m\in\bbC[\mcl]$ (so that the distribution function $N(x)$ of the spectral measure takes the value $0$ if $x<E_0$ and takes the value $1$ if $x>E_+$), then it is interesting to study the asymptotic behavior of $N(x)$ near these endpoints.  This is related to the study of not only Novikov-Shubin invariants but also the growth of the F\"{o}lner function $F(n)$ when $\nri$ \cite{GZ01} (see also \cite{KV17}) and of other group invariants.  Translated into the language of random operators, the asymptotics of $N(x)$ near $E_0$ corresponds to what is called Lifshitz tails \cite{KM07}.  For a random Schr\"{o}dinger operator on $\bbZ^d$, Lifshitz predicted
\[
N(x)\sim c_1e^{-c_2(x-E_0)^{-d/2}},
\]
which was later confirmed in some cases (for a related result, see \cite{Lukic}).  From \cite{GZ01} it follows that for random operators corresponding to an operator of convolution given by $\frac{1}{4}(a+a^{-1}+b+b^{-1})$, the function $1-N(x)$ behaves as $e^{-(x-1)^{-1}}$ near $E_+$.  It will be interesting to study what kind of asymptotic behavior $N(x)$ and $1-N(x)$ can exhibit near $E_0$ an $E_+$ respectively (as well as at any other point of the spectrum) for convolution operators given by elements $m\in\bbC[\mcl]$, in particular those outlined in Question \ref{q7}.

\medskip

\noindent\textbf{Example.}  Consider the case when $\mu>1$ so the spectrum of the operator $M_{\mu}$ acting in $\ell^2(\mcl)$ is described by part (b) of Theorem \ref{t1}.  Using Theorem \ref{t2}, we know that the corresponding spectral measure $\nu$ assigns weight $2^{-m}$ to the interval $[x_m,\mu+2/\mu]$, where $x_m$ is the largest zero of $G_m$.  It is known that $|x_m-\mu-2/\mu|$ decays exponentially, but to calculate the Novikov-Shubin invariant, we will need to know this exact rate of decay.  To do so, it suffices to find the smallest zero of $P_{m,\mu}$ (defined in \eqref{pgform}).

By the reasoning of Section \ref{bulkz}, we know that the zeros of $P_{m,\mu}$ consist of a single zero very close to $-\mu/2-1/\mu$ (call it $x_m^*$) and the remaining zeros are contained in the interval $[-2+\mu/2,2+\mu/2]$ and distribute (for $m$ large) according to the arcsine distribution on that interval.  Define $Q_m(x)=P_{m,\mu}(x)/(x-x_m^*)$.  Then it follows from (\ref{uratio}) that
\[
\left|Q_n\left(-\frac{\mu}{2}-\frac{1}{\mu}\right)\right|^{1/n}\rightarrow\mu
\]
as $\nri$.  By Poincar\'{e}'s Theorem \cite[Theorem 9.6.2]{OPUC2}, we know that
\[
\left|P_n\left(-\frac{\mu}{2}-\frac{1}{\mu}\right)\right|^{1/n}\rightarrow\frac{1}{\mu}
\]
as $\nri$.  Thus $|x_n^*+\mu/2+1/\mu|^{1/n}\rightarrow\mu^{-2}$ as $\nri$.  We conclude that, $\nu$ assigns mass $2^{-m}$ to an interval near the right endpoint of its support of length approximately $\mu^{-2m}$.  Using the definition \eqref{nsdef} we find that the Novikov-Shubin invariant in this case is
\[
\frac{\log2}{2\log\mu}=\log_{\mu}(\sqrt{2})=\frac{1}{2\log_2\mu}.
\]
Therefore, for the element $a+a^{-1}+b+b^{-1}-\mu c-\left(\mu+\frac{2}{\mu}\right)I\in\bbQ[\mcl]$ with integer $\mu\in\bbZ$ that is not a power of $2$, the ``true" Novikov-Shubin invariant given by (\ref{nsdef}) is irrational.  This recovers one of the results from \cite{G16}.

\bigskip

\noindent\textbf{Acknowledgments.}  We would like to thank D. Damanik, L. Grabowski, and R. Kravchenko for helpful discussion and remarks.  The first author graciously acknowledges support from the Simons Foundation through Collaboration Grant 527814, is partially supported by the mega-grant of the Russian Federation Government (N14.W03.31.0030), and also acknowledges the Max Planck Institute in Bonn where the work on the final part of this article was completed.

\vspace{10mm}

\end{document}